\documentclass{article}

\usepackage[english]{babel}
\usepackage[utf8]{inputenc}
\usepackage{amsmath}
\usepackage{amsthm}
\usepackage{amssymb}
\usepackage{graphicx}
\usepackage[colorinlistoftodos]{todonotes}
\usepackage{hyperref} 
\usepackage{color}
\usepackage{verbatim}

\newtheorem{theorem}{Theorem}
\newtheorem{lemma}[theorem]{Lemma}
\newtheorem{corollary}[theorem]{Corollary}

\theoremstyle{definition}
\newtheorem{definition}{Definition}

\newtheorem{remark}{Remark}

\title{Coins that Change Their Weights}
\author{Tanya Khovanova \and Konstantin Knop}
\date{}

\begin{document}

\maketitle

\begin{abstract}
As in many coin puzzles, we have several identical-looking coins, with one of them fake and the rest real. The real coins weigh the same. Our fake coin is special in that it can change its weight. The coin can pretend to be a real coin, a fake coin that is lighter than a real one, and a fake coin that is heavier than a real one. In addition to this, each time the coin is on the scale, it changes its weight in a predetermined fashion.

In this paper, we seek to find our fake coin using a balance scale and the smallest number of weighings.

We consider different possibilities for the fake coin. We discuss coins that change weight between two states or between three  states. The 2-state coin that changes weight from lighter to real and back has been studied before, so we concentrate on the 2-state coin that changes weight from lighter to heavier, and back. We also study the 3-state coin, which changes its weight from lighter to heavier to real, and back to lighter.

Given the total number of coins and the starting state of the fake coin, we calculate the smallest number of weighings needed to identify the fake coin. We provide an oblivious optimal strategy for this number of weighings. We also discuss what happens if the starting state is not known or mixed. In such cases, adaptive strategies are often more powerful than oblivious ones.
\end{abstract}

\section{Introduction}

Coin puzzles have long been a source of fascination for mathematicians. The simplest coin puzzle is formulated like this:

\begin{quote}
You are given $N$ coins that look identical, but one of them is fake and is lighter than the other coins. All real coins weigh the same. You have a balance scale that you can use to find the fake coin. What is the smallest number of weighings that guarantees your finding the fake coin?
\end{quote}

The above puzzle first appeared in 1945. Since then, there have been many generalizations of this puzzle \cite{GN}. A new generalization that allows a coin to change its weight appeared in 2015 \cite{KKP}. This generalization introduces a new type of coin, called a \textit{chameleon coin}, which can mimic a real coin or a fake coin that is lighter than a real coin. The chameleon coin has a mind of its own and can choose how to behave at any weighing. It is impossible to find chameleon coins among real coins, as the chameleons can pretend to be real all the time. An interesting question to ask would be the following: given that a mix of $N$ identical coins contains one chameleon and one classical fake coin that is lighter than a real coin, find two coins one of which is guaranteed to be the classical fake \cite{KKP}. 

The chameleon coins were further generalized in \cite{STEPAlternators}, where the \textit{alternator coins} were introduced. The alternator can mimic a real coin or a fake coin that is lighter than a real coin, but, contrary to the case of the chameleon coins, there is a deterministic rule. The alternator switches the behavior each time it is on the scale. Unlike the chameleon, the alternator coin can always be found.

In this paper we take the alternators one step further. We allow our fake coins to be heavier than real coins. We divide coins that change weight into 2-state and 3-state coins. The alternator is a 2-state coin. In this paper we also call it the LH-coin. In addition, we introduce another 2-state coin: a coin that can switch between being lighter or heavier than a real one. We call it the LR-coin.

The 3-state coins can have three states: they can pretend to be a real coin, a fake coin that is lighter than a real one, or a fake coin that is heavier than a real one. Each time it is put on a balance scale, it switches its state. The states are switched periodically.

We also study separate cases, in which we either know or do not know the starting state of the fake coin. For example, if the fake coin starts as being lighter than the real coin, and the first weighing unbalances, we know that the fake coin is on the lighter pan. If we do not know the starting state of the fake coin and the weighing unbalances, then the coin might be in the light state on the lighter pan or in the heavy state on the heavier pan. We call the resulting situation the mixed state: every coin, if it were to be found out as the fake one, is assigned a starting state. The mixed state helps us analyze the unknown starting state.

We commence with general definitions and statements in Section~\ref{sec:defs}. 

In Section~\ref{sec:lh} we study the LH-coin: the 2-state coin that changes its state between light and heavy. We find that in $w$ weighings we can process up to $3^w$ coins. We also show an oblivious strategy that can achieve this. A mixed starting state introduces an additional layer of complexity. For most distributions of the mixture of states, we can also process $3^w$ coins in the oblivious strategy. If the starting state is unknown, we can process up to $(3^w-1)/2$ coins, using an oblivious strategy to do so.

In Section~\ref{sec:alternator} we discuss the alternator, or the LR-coin: the 2-state coin that changes its state from light to real, and back. We provide an optimal oblivious strategy for a case in which the starting state of the LR-coin is known. The strategy allows for the processing of up to $J_{w+2}$ coins in $w$ weighing, where $J_n$ is the $n$-th Jacobsthal number. We also discuss the mixed state for which we can also present an oblivious strategy processing $J_{w+2}$ under certain constraints for the states. When the starting state is unknown, we can process up to $J_{w+1}$ coins in an adaptive strategy. We explain why the oblivious strategies that are as good as adaptive strategies do not exist for large $w$ when the starting state is not known. 

Next, we move to the LHR-coin: the 3-state coin that changes its weight from lighter to heavier to real, and back to lighter. The case of the 3-state coin is more complex, and so it is covered in three section. In Section~\ref{sec:3stateknown} we discuss the LHR-coin, starting in the known state. We calculate the sequence that represents the maximum number of coins that can be processed in $w$ weighings. We provide an optimal oblivious strategy for these coins. In Section~\ref{sec:3statemixed} we study the mixed state. As in previous sections, the same bound works for the number of coins in the mixed state, with some additional constraints. The mixed state allows us to find the exact bound of how many coins in the unknown state can be processed in $w$ weighings, which is done in Section~\ref{sec:3stateunknown}. We explain why the oblivious strategies that are as good as adaptive strategies do not exist for large $w$ when the starting state is not known.

\section{Definitions and General Statements}\label{sec:defs}

Let us introduce the \text{states} of the weight-changing coin while it is not on the scale. 

\begin{itemize}
\item The weight-changing coin is in the \textit{light} state if the next time it is on the scale it behaves as a fake coin that is lighter than a real coin. 
\item The weight-changing coin is in the \textit{heavy} state if the next time it is on the scale it behaves as a fake coin that is heavier than a real coin. 
\item The weight-changing coin is in the \textit{real} state if the next time it is on the scale it behaves as a real coin. 
\end{itemize}

The coins we are interested in are changing their states deterministically.

We will discuss the following types of coins:

\begin{itemize}
\item \textit{light-heavy or LH-coins} that alternate their states between light and heavy.
\item \textit{light-real or LR-coins} that alternate their states between light and real.
\item \textit{light-heavy-real or LHR-coins} that change their states from light to heavy to real, and then cycle.
\end{itemize}

Note that light-real coins were studied in \cite{STEPAlternators} and were called alternators there. Also, we do not discuss the heavy-real and heavy-light-real cases, as they can be resolved from the above cases by invoking symmetry arguments.

We also divide our research into cases on the basis of what is known of the coin's starting state. The starting state might be:

\begin{itemize}
\item \textit{known}, when we know the starting state of the coin; 
\item \textit{unknown}, when we do not know the starting state of the coin;
\item \textit{mixed}, when each coin is assigned a state, so that, should it later be found to be fake, the starting state would match the assigned state.
\end{itemize}

The mixed state might seem unusual, but it appears naturally in the study of the unknown state. Suppose some coins in the unknown state are put on the scale, and the scale unbalances. That means the fake coin can be on the lighter pan in the light state, or on the heavier pan in the heavy state. In other words, the coins that are on the scale are in the mixed state after the weighing.

During weighing strategies, some coins will be proved to be real, i.e. not fake. We call such coins \textit{genuine}, so as not to confuse them with coins in the real state.

At this point we should review with the reader the standard approach to puzzles involving fake coins and a balance scale. The balance scale has two pans; the same number of coins is put on each pan to be weighed. The output of one weighing can be described as being one of three types:

\begin{itemize}
\item ``$=$''---when the pans are equal in weight,
\item``$<$''---when the left pan is lighter,
\item ``$>$''---when the right pan is lighter.
\end{itemize}

Suppose there is a strategy that finds a fake coin in $w$ weighings. Suppose coin number $i$ is fake. Then there is a sequence of weighings after which we determine that the $i$-th coin is indeed the fake coin. We can represent the sequence of weighings that results in our conclusion that the $i$-th coin is fake as a string of three symbols: $=$, $<$, and $>$. Obviously, two coins cannot have the same string pointing to them. That means that the number of coins that can be processed in $w$ weighings is not more than $3^w$.

We call each string in the alphabet $=$, $<$, and $>$ \textit{an outcome}, as the string is a particular result of a weighing strategy. We call symbols $<$ and $>$ the \textit{unbalanced} symbols, or \textit{imbalances}.

\begin{definition}
Given an outcome $x$, \textit{the conjugate outcome}, denoted by $\bar{x}$, is the unique outcome in which all $>$'s are replaced by $<$'s, and all $<$'s are replaced by $>$'s.
\end{definition}

Note that this conjugation is an involution, as $\bar{\bar{x}} = x$. In addition, the only self-conjugate outcome of a given length is a string that consists exclusively of $=$ symbols.

Here we present our first information-theoretical bound on the number of coins that can be processed in $w$ weighings. The total number of possible outcomes for the given coin type is denoted by $T(w)$, and $S$ denotes the number of possible states.

\begin{theorem}\label{thm:itb}
If the coin starts in the known or mixed state, then the number of coins that can be processed in $w$ weighings is not more than $T(w)$. If the coin starts in the unknown state, this number is not more than $(T+S-1)/S$.
\end{theorem}

\begin{proof}
The first part follows from the fact that different outcomes correspond to different coins. If the starting state is unknown, then the only self-conjugate outcome might point to a coin that is never on the scale. We will find this coin, but we will not know its state. All other outcomes contain imbalances. As a result, when we find the fake coin, we also find its starting state. Thus, each coin is defined by $S$ different outcomes depending on its state. Therefore, if $N$ is the number of coins, the number of different outcomes must be at least $(N-1)S + 1$. The theorem follows.
\end{proof}

During any weighing, a coin's presence on the \textbf{l}eft pan is denoted by L, a coin's presence on the \textbf{r}ight pan is denoted by R, and a coin not participating (one that is left \textbf{o}utside of the weighing) is denoted by O. We call letters L and R \textit{on-scale letters}.

After all the weighings, every coin's path can be described as a string of L's, R's, and O's.

\begin{definition}
The string of L's, R's, and O's corresponding to the location of a given coin in every weighing is called the coin's \textit{itinerary}.
\end{definition}

Given an itinerary $\delta$, we denote the set of all coins with this itinerary as $\delta$, and the size of this set as $|\delta|$. We will introduce an involutive operation on itineraries, called conjugation:

\begin{definition}
Given an itinerary $\delta$, \textit{the conjugate itinerary}, denoted by $\bar{\delta}$, is the unique itinerary in which all R's are replaced by L's, and all L's replaced by R's.
\end{definition}

Note that this conjugation is an involution, as $\bar{\bar{\delta}} = \delta$. In addition, the only self-conjugate itinerary of a given length is a string of O's. After the weighings, we can partition the corpus of coins into groups by their itineraries. Given a strategy that finds a fake coin, the itinerary of each coin is uniquely defined. 

Scholars study weighing strategies of two types: \textit{adaptive} strategies in which each weighing could depend on the results of all previous weighings, and \textit{oblivious} (or non-adaptive) strategies, in which all the weighings must be specified in advance.

If we have an adaptive strategy that finds a particular fake coin, then the coin itineraries do not have to be unique. For example, if the first weighing unbalances the scale, then the coins that are not on the scale are guaranteed to be real, and we might not need to use them in the following weighings. These coins might end up with the same itinerary. With an adaptive strategy, a coin may have different itineraries depending on the outcome of the weighings. Still, each coin has a special itinerary---the itinerary of the strategy that finds this particular coin. We call this itinerary the \textit{self-itinerary}.

The oblivious strategy is different from an adaptive one. In an oblivious strategy, the itinerary of every coin is predetermined: The self-itinerary is the only itinerary. 

\begin{lemma}
Distinct coins have distinct self-itineraries. 
\end{lemma}

\begin{proof}
If the strategy is oblivious, and two coins have the same itineraries, then they are always together in the same pile, in all the weighings. If the fake coin is one of them, we cannot identify it.

Suppose two coins have the same self-itineraries in an adaptive strategy, and one of them is fake. Then, in the first weighing, these two coins have to be in the same pile (left pan, right pan, outside). The result of the weighing is the same whether the first or the second coin is fake. The second weighing is uniquely defined by the results of the first weighing; therefore, both coins follow the same strategy in the second weighing. They will be in the same pile again, and so on. After all the weighings, both coins will always be together, and we cannot say which of them is fake.
\end{proof}

The self-itinerary should match the outcome corresponding to the coin. If the coin is ever on the scale according to its self-itinerary, then, if this coin is fake, once we find it, we will also know the state of this coin during every weighing. This knowledge means that the outcome is uniquely determined per this self-itinerary. Also, if the coin is on the scale in the unbalanced state, the outcome of this particular weighing must match the state.

If the starting state is unknown, in an oblivious strategy a coin has the itinerary that corresponds to $S$ different states. In an adaptive strategy, the first weighing is defined uniquely; that means the same coin in different starting states has the same first letter in its self-itinerary.

\subsection{Oblivious strategies}

Let us consider an oblivious strategy, in which every coin by definition has its predetermined itinerary. Now, from this set of itineraries, we generate a weighing strategy. In weighing number $i$ we consider the $i$-th letter in every itinerary string. If this letter is L, the corresponding coin is put on the left pan. If this letter is R, the corresponding coin is put on the right pan. If this letter is O, the corresponding coin is not put on the scale at all. We also denote via $\textup{Left}_i$, $\textup{Right}_i$ and $\textup{Out}_i$ the sets of coins that are on the left pan, right pan, and outside, in the $i$-th weighing.

The following statement is standard. It explains when we can produce a legitimate weighing strategy from the set of itineraries for the coins given.

\begin{lemma}
The set of itineraries can generate a legitimate strategy if and only if $|\textup{Left}_i|=|\textup{Right}_i|$, for every index $i$.
\end{lemma}

\begin{proof}
The condition guarantees that at each weighing the number of coins that are put on each pan is the same. 
\end{proof}

\begin{corollary}\label{cor:scstr}
If the set of itineraries is self-conjugate, then it corresponds to a legitimate strategy.
\end{corollary}

In finding an oblivious strategy, our approach is as follows. We find the condition of the weighing strategy that limits what kind of outcome strings in the alphabet $=$, $<$, and $>$ can lead to finding the fake coin. We assign the outcome strings to the coins. Then we build itineraries for the given coins, which correspond to their outcomes. After that, we prove that the itineraries describe an oblivious weighing strategy that works.

\section{A Light-Heavy Coin}\label{sec:lh}

\subsection{Starting state is known}\label{sec:lhknown}

As we have mentioned above, the total number of coins is $N$, and we have one fake coin of type light-heavy that we would like to find. Without loss of generality, we can assume that the starting state of the coin is light.

We will show that if $3^{w-1} < N \leq 3^w$, then the optimal strategy finds the fake coin in $w$ weighings. We cannot do better than that, as the number of possible strings of length $w$ in the alphabet $=$, $<$, and $>$ is not more than $3^w$. 

Now we assign the outcomes to our coins. If the number of coins is odd, we pick one of the coins and assign the self-conjugate outcome to it. For an even number of coins, we assign outcomes to coins in conjugate pairs. 

Next we want to translate the outcomes to itineraries. An outcome uniquely defines the itinerary of the fake coin that corresponds to this outcome. This is true because, in the case of the light-heavy coin, each weighing tells us exactly in which of the three piles the fake coin is.

We call an odd(even) occurrence of an imbalance an \textit{odd(even) imbalance}. For example, in the outcome $=<>=>$, the second and fifth symbols are odd imbalances, while the third symbol is an even imbalance. This is how we translate outcomes to itineraries.

\begin{itemize}
\item An odd imbalance: $<$ is replaced by L and $>$ is replaced by R.
\item An even imbalance: $<$ is replaced by R and $>$ is replaced by L.
\item The equality sign is replaced by O.
\end{itemize}

A set of itineraries provides us with an oblivious weighing strategy.

\begin{theorem}
If $3^{w-1} < N \leq 3^w$, then there exists an oblivious optimal strategy that finds the fake coin in $w$ weighings. 
\end{theorem}

\begin{proof}
Conjugate outcomes correspond to conjugate itineraries. As the set of assigned outcomes is self-conjugate, the set of itineraries is self-conjugate as well. By Corollary~\ref{cor:scstr} we can generate a legitimate set of weighings.

Consider an outcome $x$ of this strategy. If a weighing is balanced, the fake coin must not be on the scale. If a weighing is an odd imbalance, then the fake coin is on the lighter pan. That means if the imbalance is $<$, the fake coin is on the left pan;  otherwise, it is on the right pan. For an even imbalance, the opposite is true. That means the fake coin must have the itinerary as assigned by the rule above. As this is the only coin with this itinerary, the strategy finds it.
\end{proof}

Note that the strategy matches the information-theoretical bound (see Theorem~\ref{thm:itb}). This means that there is no adaptive strategy prescribing fewer weighings than the described oblivious strategy.

We see that this problem is very similar to the problem of finding one fake coin that is lighter than the real coin. Given the total number of coins, the same set of itineraries can provide oblivious strategies for both cases.

\subsection{Starting state is mixed}\label{sec:lh-mixed}

Suppose all coins are divided into two disjoint groups: $l$ coins, such that, if the fake coin is there, it must be in the light state; $h$ coins, such that, if the fake coin is there, it must be in the heavy state. We call such groups of coins the \textit{mixed-known} state, or $l:h$ state. 

As before, an outcome uniquely defines the self-itinerary. It follows that the coins in the $l$-group, as well as the coins in the $h$-group, must have distinct self-itineraries. Two coins in different starting states can, however, have the same self-itinerary; yet they cannot have conjugate self-itineraries.

That means that if at least one of $l$ or $h$ is even, we can produce an oblivious strategy. We do this by assigning pairs of coins in the same group to conjugate pairs of itineraries. If we have an extra coin, we assign a self-conjugate itinerary to it. Thus, we have proven the following lemma.

\begin{lemma}
If $lh$ is even and $l+h \leq 3^w$, there exists an oblivious strategy that finds the fake coin in the mixed $l:h$ state.
\end{lemma}

Suppose $l$ and $h$ are both odd. To begin with, we should note that the $1:1$ state is unsolvable. Without loss of generality, let us assume that $h > 1$. Consider an example of a $1:3$ state. It can be solved in two weighings in an oblivious strategy by assigning the itinerary LO to the $l$-coin, and itineraries LO, RL, and RR to the $h$ coins. 

We will expand this example to solve the $l:h$ state in $w$ weighings for $l+h \leq 3^w-3$ and $l$ and $h$---both odd. Namely, we assign LOOOO$\ldots$ to the $l$-coins and LOOOO$\ldots$, RLOOO$\ldots$, and RROOO$\ldots$ to three $h$-coins. Next, we remove the itineraries that we used and their conjugates from consideration. So far, we have assigned itineraries to 4 coins and have removed 6 itineraries. We divide all the other itineraries into conjugate pairs, and assign conjugate pairs to two coins in the same group.

\begin{lemma}
Suppose $l+h \leq 3^w$, except the case when both $l$ and $h$ are odd and $l+h= 3^w -1$; then there exists an oblivious strategy that finds the fake coin in the mixed $l:h$ state. If both $l$ and $h$ are odd and $l+h= 3^w -1$, neither an adaptive nor an oblivious strategy that finds the coin in $w$ weighings exists.
\end{lemma}

\begin{proof}
What is left to show is that for the case $l+h = 3^w-1$, and $l$ and $h$ both odd, an adaptive strategy is impossible. Consider the first weighing. We have to put $3^{w-1}$ coins on each pan; otherwise, if the weighing yields a balance, we would have too many coins left. Suppose we have $l_1$ and $l_2$ coins from the $l$-group on the first and the second pan correspondingly, and, similarly, $h_1$ and $h_2$ for the $h$-group. If the left pan is lighter, then the fake coin is either one of the $l_1$ coins in the light state on the left pan, or one of the $h_2$ coins in the heavy state on the right pan. It follows that $l_1+h_2 \leq 3^{w-1}$. Similarly, $l_2+h_1 \leq 3^{w-1}$. Therefore, $l_1=l_2$ and $h_1=h_2$. If the weighing balances, then we have used an even number of coins in each group and are left with a situation where we can invoke induction. The lemma follows from the fact that the $1:1$ state is unsolvable.
\end{proof}

\subsection{Starting state is unknown}\label{sec:lh-unknown}

We start with the following lemma.

\begin{lemma}
In the oblivious strategy of finding one light-heavy coin that starts in the unknown state, we cannot have two coins with conjugate itineraries.
\end{lemma}

\begin{proof}
Suppose we have two conjugate itineraries assigned to two different coins. Suppose one of these coins ends up being fake. The second coin is then always opposite the first coin, and is on the scale at least once. That means the second coin might also be fake, starting in the opposite state.
\end{proof}

The problem of finding the light-heavy coin is similar to another classical coin-weighing problem, in which we need to find one coin that might be heavier or lighter. In this calssical problem, an oblivious strategy cannot have two conjugate itineraries assigned to two different coins. Indeed, if one coin in the pair is found to be fake and lighter, then the other coin might also be fake, but heavier.

There are many papers that explain an oblivious strategy for this classical problem, for 12 and 13 coins, that can be solved in 3 weighings \cite{Schell47}. The bounds for any number of weighings are done in \cite{Fine}. The oblivious optimal strategy for any number of coins is in \cite{Dyson,Smith}.

\begin{lemma}
We can find the light-heavy coin starting in an unknown state in $w$ weighings if the number of coins is not more than $(3^w-1)/2$.
\end{lemma}

The proof is the same as in the classical case \cite{Fine}. We will repeat it, as we use a similar reasoning later.

\begin{proof}
Suppose conjugate pairs of outcomes are assigned to coins. The sole self-conjugate outcome may be assigned to only one coin. This way, we can process at most $(3^w+1)/2$ coins in $w$ weighings. This is the same bound as the bound in Theorem~\ref{thm:itb}. Another consideration is parity. Out of $(3^w+1)/2$ conjugate pairs of outcomes, exactly $(3^{w-1}+1)/2$ start with $=$. That means an odd number of coins, namely $3^{w-1}$, corresponds to pairs of outcomes that start with an imbalance. These are precisely the coins that will be put on the scale in the first weighing. Yet we have to use an even number of coins. That means that we need to throw out at least one of the conjugate pairs. Thus, we can process no more than $(3^w-1)/2$ coins in $w$ weighings.
\end{proof}

In the classical problem, we know that any number of coins up to the theoretical maximum above can be processed in $w$ weighings \cite{Dyson, Smith}. This is done by picking one itinerary from a conjugate pair in such a way that the resulting set of itineraries creates a legitimate weighing strategy. We can use the same set of itineraries to solve our light-heavy problem.

\begin{theorem}
If the number of coins $N$ is in the range $(3^{w-1}-1)/2 < N \leq (3^w-1)/2$, then the light-heavy coin with an unknown starting state can be found in $w$ weighings, using an oblivious strategy.
\end{theorem} 

\begin{proof}
First, we assign itineraries to coins in the same way as in the classical problem: no more than one itinerary from a conjugate pair is assigned, and the itineraries balance each other. This means that we have a legitimate weighing strategy. Suppose this strategy produces an outcome $x$. If the coin started in the light state, then its itinerary $\delta$ must be uniquely defined by the outcome via rules in Section~\ref{sec:lh-unknown}. If it started in the heavy state, the itinerary must be $\bar{\delta}$. As not more than one of these itineraries were used, the fake coin is found.
\end{proof}

\section{The Light-Real Coin}\label{sec:alternator}

In this section we discuss the light-real coin. This coin was studied in \cite{STEPAlternators}, where it was called an alternator, and optimal adaptive strategies for known and unknown starting states were proposed. We provide a summary of those results for a complete picture, as well as an oblivious strategy for these states. In addition, we discuss the mixed state.

The outcomes for this case have the following property: imbalances cannot follow each other. It follows that the number of possible outcomes of lengths $w$ is $J_{w+2}$, where $J_w$ is a Jacobsthal number. 

Jacobsthal numbers are defined as a sequence with a recursion: $J_{w+1}=J_w+2J_{w-1}$, and initial conditions $J_0=0$, $J_1=1$. The Jacobsthal sequence grows approximately as a power of 2: $J_{n+1}=2J_n + (-1)^n$ and $J_n = (2^n - (-1)^n)/3$.

\subsection{Starting state is known}

If the coin starts in the light state, then the number of possible outcomes of length $w$ is bounded by $J_{w+2}$. If the coin starts in the real state, then the first symbol of an outcome must be $=$ and the number of possible outcomes is bounded by $J_{w+1}$.

The paper \cite{STEPAlternators} provides an adaptive strategy that shows that the bound described here is exact. In this paper we build itineraries for the coins that describe an oblivious strategy with the same bound. 

The interesting new thing that happens here, as compared to the previous sections, is that the itineraries are not uniquely defined by the outcome. If the outcome has an imbalance for a weighing, then the pile where the fake coin must be is uniquely defined. If the symbol is $=$, then the pile is not defined uniquely. The condition is that the coin has to be on the scale exactly once between two unbalanced weighings. Also, if the starting state is real, the fake coin must appear on the scale before the first unbalanced weighing.

Let us suppose that the starting state is light. This is how we build our itineraries: we say that the coin is on the scale in weighing $i$ if the $i$-th symbol in the outcome is $=$ and the previous symbol is unbalanced. For this weighing, we put the coin on the same pan that held the coin in the previous weighing.

Here is the exact rule to create a letter in the itinerary in the $i$-th place from the $i$-th and $(i-1)$st symbol in the outcome:

\begin{itemize}
\item If the outcome has $<$ (correspondingly $>$) in the $i$-th place, the itinerary has L (correspondingly R).
\item If the outcome has $=$ in the $i$-th place and the previous place either does not exist or also has $=$, the itinerary has O.
\item If the outcome has $=$ in the $i$-th place and $<$ in place $i-1$ (correspondingly $>$ in place $i-1$), the itinerary has L (correspondingly R).
\end{itemize}

Note that we can describe the itineraries in such a way that every on-scale letter (L or R) appears in consecutive pairs, unless it is the last one.

Suppose the starting state is real. Then, we need to place the fake coin on the scale exactly once, before the first unbalanced weighing occurs. As a reminder, the first outcome is always balanced, and therefore every unbalanced weighing has a previous balanced weighing. And so, we place the fake coin on the scale for a balanced weighing before each imbalance. Here is the exact rule to create a letter in the itinerary in the $i$-th place from the $i$-th and $(i+1)$-st symbol in the outcome:

\begin{itemize}
\item If the outcome has $<$ (correspondingly $>$) in the $i$-th place, the itinerary has L (correspondingly R).
\item If the outcome has $=$ in the $i$-th place and the next place either does not exist or also has $=$, the itinerary has O.
\item If the outcome has $=$ in the $i$-th place and $<$ in place $i+1$ (correspondingly $>$ in place $i+1$), the itinerary has L (correspondingly R).
\end{itemize}

Note that we can describe the itineraries in such a way that every on-scale letter appears in consecutive pairs.

It is easy to see that conjugate outcomes generate conjugate itineraries per these rules. We use these itinerary assignments to prove our theorem.

\begin{theorem}
If the number of coins $N$ is between Jacobshtal numbers: $J_{w+1} < N \leq J_{w+2}$, then there exists an oblivious strategy that finds the light-real coin in $w$ weighings if it starts in the light state, and in $w+1$ weighings if it starts in the real state.
\end{theorem}

\begin{proof}
Given the total number of coins, we assign the outcomes to coins similar to the method we used for the light-heavy coin with a known starting state in Section~\ref{sec:lhknown}. If the number of coins is odd, then one of the coins is matched to the self-conjugate outcome. The rest of the coins are assigned in pairs to conjugates outcomes. We build itineraries from the outcomes as described above. This way, we are guaranteed to have itineraries in conjugate pairs, and therefore a corresponding strategy exists.

Suppose this strategy produces an outcome $x$; then the corresponding self-itinerary must be uniquely defined in places where the outcome is unbalanced. Since we know the starting state of the coin and the rules we created for building possible itineraries, we know the whole self-itinerary and, consequently the light-real coin.
\end{proof}

\subsection{Starting state is mixed}\label{sec:lr-mixed}

Suppose all coins are divided into two disjoint groups: $l$ coins such that, if the fake coin is there, it must be in the light state, $r$ coins such that, if the fake coin is there, it must be in the real state. We call such groups of coins the \textit{mixed-known} state, or $l:r$ state. 

We already know from counting the outcomes that, to process $N$ coins in $w$ weighings, we need the following inequalities: $r \leq J_{w+1}$ and $l+r \leq J_{w+2}$. This means that we can assign outcomes to coins in the following manner. First, we assign the outcomes that start with $=$ to the coins in the $r$-group, and then other outcomes to the coins in the $l$-group. Moreover, we strive to assign pairs of conjugate outcomes to the same group. If the number of coins in the $r$-group is odd, we assign the self-conjugate outcome to one coin in the $r$-group, and other outcomes, starting with the balance in conjugate pairs, to other coins in the $r$-group. If the number of coins in the $r$-group is even, then the total number of coins in the group is less than the total number of outcomes, starting with the balance. In this case, we assign conjugate pairs of outcomes to the coins in the $r$-group, and leave the self-conjugate outcome for a coin in the $l$-group, if needed. After all the coins in the $r$-group are assigned, we assign the outcomes to the coins in the $l$-group in conjugate pairs. 

The only case left is when $l$ and $r$ are odd. In this case, we assign the self-conjugate outcome to a coin in the $r$-group, and the outcome $<===\ldots$ to a coin in the $l$-group. The other outcomes we assign in conjugate pairs.

We will need to match outcomes to itineraries. Unlike what we saw before with other types of coins, an outcome does not define the self-itinerary uniquely. For example, coins in the real state with itineraries LO, RO, OL, OR, and OO all have the same outcome $==$. The good news is that, given an itinerary and the state of a coin, the outcome that leads to this coin is defined uniquely.

\begin{theorem}
There is an oblivious strategy that solves the $l:r$ case in $w$ weighings, as long as $r \leq J_{w+1}$ and $l+r \leq J_{w+2}$.
\end{theorem}

\begin{proof}
We already described the assignment of the outcomes above. Now we need to assign itineraries. If $lr$ is even, then we match the self-conjugate outcome (if it is assigned) with the self-conjugate itinerary. Other outcomes are assigned in conjugate pairs to the coins in the same group. That means we can assign to the two coins a pair of conjugate itineraries that match the pair of conjugate outcomes.

In a special case of $lr$ being odd, we assign the itinerary LOOO$\ldots$ to the extra coin in the $l$-group and ROOO$\ldots$ to the extra coin in the $r$-group. The sets of itineraries are paired in such a way that they correspond to a legitimate oblivious weighing strategy.

Suppose we use this strategy and get an outcome. Each coin has its own starting state and, together with the itinerary, matches an outcome uniquely. That means only one coin can match a given outcome, and this coin is found.
\end{proof}

\subsection{Starting state is unknown}

We know that there are $J_{w+2}$ possible outcomes of $w$ weighings. The light-real coin starting in the real state will have an outcome starting with symbol $=$. This means that the number of possible coins we can process is not more than $J_{w+1}$.

This bound is precise for an adaptive strategy, as was shown in \cite{STEPAlternators}:

\begin{lemma}
If the number of coins $N$ is in the range $J_w < N \leq J_{w+1}$, then the light-real coin with an unknown starting state can be found in $w$ weighings, using an adaptive strategy.
\end{lemma}

Here we would like to discuss ideas of how to build an oblivious strategy. Suppose such a strategy existed; then, every coin would have an itinerary. In addition, every outcome could be matched to an itinerary and state of a coin. Table~\ref{table:2wI2c} describes this matching for two weighings. The word 'light/real' is placed in a cell if the coin starting in the light(real) state can have the corresponding itinerary with the corresponding outcome.

\begin{table}[h!]
\centering
\begin{tabular}{| c | c | c | c | c | c |} 
 \hline
  & $<=$ & $=<$ & $==$ & $=>$ & $>=$ \\ 
 \hline
 LL & light & real & & & \\ 
 LO & light & & real & & \\ 
 LR & light & & & real & \\ 
 OL & & light & real & & \\ 
 OO & & & light real & & \\ 
 OR & & & real & light & \\ 
 RL & & real & & & light \\ 
 RO & & & real & & light \\ 
 RR & & & & real & light \\ 
 \hline
\end{tabular}
\caption{Matching of itineraries for 2 weighings}
\label{table:2wI2c}
\end{table}

If we know the itinerary and the starting state, the outcome is uniquely defined. That means each row has exactly one of each word: `light' and `real'. These two words share a cell in the table if and only if the cell matches the self-conjugate itinerary with the self-conjugate outcome.

We can represent this table as a bipartite graph with two-colored edges. Vertices correspond to outcomes and itineraries, while edges are colored according to the initial state. Two vertices are connected by an edge of a particular color if the corresponding word is at the intersection of the row and column matching this outcome and itinerary. Choosing a strategy means choosing a subset of itineraries. That is, we pick an induced subgraph corresponding to this set of itineraries. In this subgraph, an outcome must be connected to no more than one itinerary.

In our example, we want to find an oblivious strategy for three coins, as this is the theoretical maximum. This means that we need at least five outcomes, each connected to a different itinerary. We must use the $==$ outcome which points to the OO itinerary. That means we cannot use LO, OL, OR, or RO itineraries, as they produce the $==$ outcome while in the real state. We can use no more than one itinerary from the pair LL---LR, as they share the first column. Similarly, we can only use one from RL---RR. We can build our itineraries in two ways: LL, OO, RR, or LR, OO, RL.

Let us introduce another nomenclature here. We will use the $\neq$ sign to represent an imbalance $<$ or $>$. If a coin corresponds to the outcome $\neq =$, then it has to be in the light state. Then, in the real state, the same coin must have outcome $= \neq$ or $==$. The $==$ outcome is already assigned to a coin that is never on the scale. That means we need to match $\neq =$ to $=\neq$ for the same coin. If these two outcomes represent the same coin in different states, then the coin does not have the letter O in its itinerary. For simplicity's sake, we match the second letter to the first in an itinerary. So, we pick LL, OO, RR.

We present the solution in Table~\ref{table:2wi2o}.

\begin{table}[h!]
\centering
\begin{tabular}{| c | c | c |} 
 \hline
  & light & real\\ 
 \hline
 LL & $<=$ & $=<$ \\ 
 OO & $==$ & $==$ \\ 
 RR & $>=$ & $=>$ \\
 \hline
\end{tabular}
\caption{Matching of itineraries to outcomes for two weighings}
\label{table:2wi2o}
\end{table}

The important thing here is that different outcomes correspond to different coins, and, with one exception, an outcome uniquely defines a coin's state as well.

Let us use this language to find an oblivious strategy for three weighings. We have the following groups of outcome patterns: $\neq = \neq$, $\neq ==$, $= \neq =$, $== \neq$, $===$. If a coin corresponds to the outcome $\neq = \neq$, then it has to start in the light state, and must always be on the scale. There are 4 possible outcomes like this. If the same coin starts in the real state, then the outcome must be $= \neq =$. There are only 2 possible outcomes like this. We need to pair two itineraries (for the light and real starting state) to the same coin. This means we can only match these itineraries to two coins. Let us choose itineraries LLL and RRR for these outcomes. 

We can match the outcome pattern $\neq ==$ for the light starting state to $== \neq$ for the real starting state. The corresponding itineraries have to have O in second place, and, for consistency, we will use a repeated letter for an unbalanced weighing. The itineraries are LOL and ROR. There is also the OOO itinerary.

Let us look at the matching Table~\ref{table:3wi2o}:

\begin{table}[h!]
\centering
\begin{tabular}{| c | c | c |} 
 \hline
  & light & real\\ 
 \hline
 LLL & $<=<$ & $=<=$ \\ 
 LOL & $<==$ & $==<$ \\ 
 OOO & $===$ & $===$ \\ 
 ROR & $>==$ & $==>$ \\
 RRR & $>=>$ & $=>=$ \\
 \hline
\end{tabular}
\caption{Matching of itineraries to outcomes for three weighings}
\label{table:3wi2o}
\end{table}

Thus, we get an oblivious weighing strategy for three weighings and 5 coins. If we number the coins corresponding to rows in Table~\ref{table:3wi2o}, the first weighing must compare coins 1 and 2 versus coins 4 and 5. The second weighing compares coins 1 and 5, and the third weighing is the same as the first one. If we want to process four coins in three weighings, we can use the same strategy, ignoring the third coin.

Let us move to four weighings. We start with matching patterns in Table~\ref{table:4wop}:

\begin{table}[h!]
\centering
\begin{tabular}{| c | c |} 
 \hline
light & real \\
\hline
$\neq = \neq =$ & $= \neq = \neq$ \\ 
$\neq == \neq$ & $=\neq ==$ and $==\neq =$ \\ 
$\neq ===$ & $===\neq$ \\ 
$====$ & $====$ \\ 
 \hline
\end{tabular}
\caption{Outcome pattern matching for four weighings}
\label{table:4wop}
\end{table}

We will assign itineraries so that if an itinerary of a coin matches the pattern in the first column for the light state, the same coin will have the pattern in the same row in the next column for the real state. Now we need to assign itineraries so that they balance. For the first row in Table~\ref{table:4wop}, we get: LLLL, LLRR, RRLL, and RRRR. For the second row: LOLL, RORR, LLOL, and RROR. For the third row: LOOL and ROOR. In the last row we get OOOO. Note that the first and third groups have itineraries in conjugate pairs.

The final matching of itineraries to outcomes is in Table~\ref{table:4wi2o}.

\begin{table}[h!]
\centering
\begin{tabular}{| c | c | c |} 
 \hline
  & light & real\\ 
LLLL    &      $<=<=$    &      $=<=<$ \\
LLRR     &     $<=>=$    &      $=<=>$ \\
RRLL     &     $>=<=$    &      $=>=<$ \\
RRRR     &     $>=>=$    &      $=>=>$ \\
LOLL     &     $<==<$    &      $==<=$ \\
RORR     &     $<==>$    &     $ ==>=$ \\
LLOL     &     $>==<$    &      $=<==$ \\
RROR     &     $>==>$    &      $=>==$ \\
LOOL    &      $>===$    &      $===<$ \\
ROOR    &      $<===$    &      $===>$ \\
OOOO   	 &     $====$    &      $====$ \\
 \hline
\end{tabular}
\caption{Matching of itineraries to outcomes for four weighings}
\label{table:4wi2o}
\end{table}

Note that because we have one self-conjugate itinerary and all other itineraries are in conjugate pairs, we can have an oblivious strategy for any number of coins below 11.

So far, we have built an oblivious strategy that is as powerful as an adaptive strategy for any number of weighings up to and including four. When we move to five weighings, the situation changes. We have 43 possible outcomes and, as was shown in \cite{STEPAlternators}, the maximum possible number of coins that can be processed in an adaptive strategy is 21. An adaptive strategy for any number of coins up to 21 was also described in the same paper \cite{STEPAlternators}. 

We will now prove that we cannot have an oblivious strategy with more than 20 coins. 

Consider the pattern $\neq = \neq = \neq$. There are 8 outcomes with this pattern. They must correspond to coins starting in the light state and going on the scale all 5 times. These coins in the real state will have an outcome with pattern $= \neq = \neq =$. There are only 4 possible outcomes like this. Therefore, we can only match all these outcomes to no more than 4 coins. Thus, 4 out of 8 outcomes of pattern $\neq = \neq = \neq$ must be unmatched. We will have 39 outcomes left, which can be matched to no more than 20 coins.

We describe an explicit oblivious strategy for 20 coins in Appendix~\ref{app:al5w20c}. By subtracting the necessary number of conjugate itineraries, or the self-conjugate itinerary, we get an oblivious strategy for any number of coins below 20.

There will be a similar problem with more weighings when the number of weighings is odd. For $2k+1$ weighings, there are $2^{k+1}$ alternating outcomes that start with an imbalance, versus $2^k$ alternating outcomes that start with a balance. Thus $2^k$ outcomes cannot be matched. We also have to subtract $2^{k-1}$ coins from the information-theoretical bound. That means that for $2k+1$ weighings the number of coins that can be processed in an oblivious strategy is not more than $(J_{2k+3} - 2^k +1)/2$.

What happens if the number of weighings is even? A similar problem arises. For 6 weighings, the outcome patterns $\neq = \neq == \neq$ and $\neq == \neq = \neq$ can only be for coins starting in the light state. The same coins in the real state can have outcomes $= \neq = \neq ==$, $= \neq == \neq =$, and $== \neq = \neq$. But there are 16 outcomes on the one hand versus 12 on the other.

For an even number of weighings, $2k$, the number of outcomes that start and end with an imbalance and have $k$ imbalances is $(k-1)2^k$. The corresponding outcomes for the real state must have $(k-1)$ imbalances, and have to start and end with balances. There are $k2^{k-1}$ of them. The difference is $(k-1)2^k - k2^{k-1} = (2k-2 - k) 2^{k-1} = (k-2) 2^{k-1}$. This means that for an even number of weighings---that is at least 4---the number of coins that can be processed in an oblivious strategy is not more than $(J_{2k+2} - (k-2) 2^{k-1} +1)/2$.

We see that, for the unknown state---at five weighings and up---oblivious strategies are less powerful than adaptive strategies.

According to our new adjusted bound, the number of coins that can be processed in $n$ weighings---where $n$ starts from 0---is not more than the following:

\[1,\ 1,\ 3,\ 5,\ 11,\ 20,\ 41,\ 82,\ 163,\ \ldots.\]

An oblivious strategy for 6 weighings and 41 coins is presented in Appendix~\ref{app:al6w41c}, and for 7 weighings and 82 coins---in Appendix~\ref{app:al7w82c}.

\section{A Light-Heavy-Real Coin. Known Starting state.}\label{sec:3stateknown}

We are given $N$ identical-looking coins. All but one coin are real and weigh the same. One coin is special, and is called \textit{the 3-state coin}. It can change its weight to mimic three different types of coins: a real coin, a fake coin that is lighter than a real one, and a fake coin that is heavier than a real one. The 3-state coin switches its behavior in a periodic way. The pattern can be light-heavy-real and so on, or heavy-light-real and so on. Invoking symmetry, we can study only one of two patterns. Let us say that our 3-state coin switches its weight from lighter to heavier to real, and back to lighter. We call this coin the LHR-coin.

\subsection{Outcomes}

Our method is to match outcomes to coins. Similarly to light-real coins, not every outcome is possible. If our coin is present on the scale in the heavy state, then the next weighing must balance. Indeed, our coin will be in the real state. Therefore, whether or not our coin is on the scale, the scale will balance. 

As before, we consider patterns of outcomes where we replace imbalance symbols $<$ and $>$ with an inequality sign.

\begin{lemma}
The outcome patterns must follow the following rules, depending on the starting state of the LHR-coin and the parity of an imbalance in the outcome string:
\begin{itemize}
\item If the starting state is light, then every even imbalance that is not at the end must be followed by $=$.
\item If the starting state is heavy, then every odd imbalance that is not at the end must be followed by $=$.
\item If the starting state is real, then every even imbalance that is not at the end must be followed by $=$. In addition, the first symbol of the outcome must be $=$.
\end{itemize}
\end{lemma}

The possible number of outcomes of length $n$ with these properties provides an upper bound for the number of coins that can be processed in $n$ weighings. Let us denote the sets of possible outcomes for light, heavy, and real starting states of length $n$ as $\mathcal{L}_n$, $\mathcal{H}_n$, and $\mathcal{R}_n$ correspondingly. We also denote the total number of such outcomes as $L_n$, $H_n$, and $R_n$ correspondingly.

We can express the total number of outcomes for $n$ through the number of outcomes for $n-1$.

\begin{lemma}\label{lemma:lhr-recursion}
$$L_n = L_{n-1} +2 H_{n-1},$$
$$H_n = H_{n-1} +2 R_{n-1}.$$
$$R_n = L_{n-1}.$$
\end{lemma}

\begin{proof}
If an $\mathcal{L}$-outcome starts with a balance, then the rest of the outcome string must be an $\mathcal{L}$-outcome. If it starts with an imbalance, then the coin switches to the heavy state and the rest of the outcome string must be an $\mathcal{H}$-outcome.

If an $\mathcal{H}$-outcome starts with a balance, then the rest of the outcome string must be an $\mathcal{H}$-outcome. If it starts with an imbalance, then the coin switches to the real state and the rest of the outcome string must be an $\mathcal{R}$-outcome.

An $\mathcal{R}$-outcome cannot start with an imbalance. The rest of the outcome string must be an $\mathcal{L}$-outcome.
\end{proof}

Initially, we have $L_0 = H_0 = R_0 = 1$. The lemma allows us to calculate the sequences $L_n$, $H_n$, and $R_n$:

\begin{itemize}
\item $L_n$: 1, 3, 9, 19, 41, 99, 233, 531, $\ldots$.
\item $R_n$: 1, 1, 3, 9, 19, 41, 99, 233, $\ldots$.
\item $H_n$: 1, 3, 5, 11, 29, 67, 149, 347,  $\ldots$.
\end{itemize}

We can express $L_n$, $H_n$, and $R_n$ as a third-order recursion.

\begin{lemma}
Sequences $L_n$, $H_n$, and $R_n$ follow the same recursion: $s_{n+1} = 2s_n - s_{n-1} + 4s_{n-2}$.
\end{lemma}

\begin{proof}
From $L_n = L_{n-1} + 2H_{n-1}$, we get $2H_{n-1} = L_n - L_{n-1}$. Therefore, $4L_{n-2} = 4R_{n-1} = 2(H_{n+1} - H_n) = 
L_{n+1} - L_n -L_n + L_{n-1}$. That is 
$$L_{n+1} = 2L_n - L_{n-1} +4L_{n-2}.$$
Sequences $H_n$ and $R_n$ are linear combinations of the sequence $L_n$ with a shifted self. Therefore, they follow the same recursion.
\end{proof}

These sequences are called \textit{weighted Tribonaccis} with weights $(2,-1,4)$.

The approximate values of the roots of the weighted Tribonacci equation $x^3 = 2x^2 - x + 4$ are $ x = 2.3146, x = -0.157298 + 1.30515 i, x = -0.157298 - 1.30515 i.$ This gives us a growth estimate for the number of outcomes that grow approximately as powers of $2.3146$. Not surprisingly, this number is between 2 and 3. The number of coins that can be processed with one standard fake coin growth as a power of 3. Our case is worse than that. The number of coins that can be processed with one LR-coin grows as a power of 2. Our case is better than that. Indeed, when a coin in a real state is on the scale, we do not get any useful information; we just change the state of the coin. In a 3-state coin, the real state happens less often than in the case of the LR-coin.

For future reference, we would like to note the relative values of these sequences:

\begin{lemma} 
For $n > 1$
\[ R_n < H_n < L_n.\]
\end{lemma}

\begin{proof}
The proof is by induction, where the initial step is verified by observing the starting elements of these sequences above. Now, assume that the statement is true for $n-1$. It follows that $L_n = L_{n-1} + 2H_{n-1} > H_{n-1} +2R_{n-1} = H_n$. Analogously, $H_n=H_{n-1} + 2R_{n-1} > 3R_{n-1} = 3L_{n-2} > L_{n-2} + 2H_{n-3} = L_{n-1} = R_{n-1}$.
\end{proof}

We would also like to mention a powerful tool that allows us to prove linear inequalities relating to $L_n$ or $H_n$. Suppose $A$ is a vector in a $k$-dimensional space. For $n \geq k$ we denote as $A(L,n)$ (correspondingly $A(H,n$)) a dot product of $A$ with $(L_n, L_{n-1},\ldots,L_{n-k+1})$ (correspondingly $(H_n, H_{n-1},\ldots,H_{n-k+1})$.

\begin{lemma}\label{lemma:AB}
If there exists a value $m > k$, for which it is true that $A(L,m)\ + B(H,m)\ ?\ 0$ and $A(L,m-1)\ + B(H,m-1)\ ?\ 0$ and $A(H,m) + B(L,m-1)\ ?\ 0$, where ? represents the same equality sign or inequality sign, then for any $n >m$ it is true that $A(L,n)\ + B(H,n)\ ?\ 0$ and $A(H,n) + B(L,n-1)\ ?\ 0$.
\end{lemma}

\begin{proof}
The proof is by induction. If the statement is true for $n-1$, then $A(L,n)\ + B(H,n) = A(L,n-1) + 2A(H,n-1) + B(H,n-1) + 2B(L,n-2) = (A(L,n-1) + B(H,n-1)) + 2(A(H,n-1) +  B(L,n-2))\ ?\ 0$. Similarly, $A(H,n) + B(L,n-1) = A(H,n-1) + 2A(L,n-2) + B(L,n-2) + 2B(H,n-2) =  (A(H,n-1) + B(L,n-2)) + 2(A(L,n-2) + B(H,n-2))  \ ?\ 0$.
\end{proof}

By putting $B=0$, we get the following corollary.

\begin{corollary}\label{cor:A}
If there exists a value $m > k$ for which it is true that $A(L,m)\ ?\ 0$ and $A(L,m-1)\ ?\ 0$ and $A(H,m)\ ?\ 0$, where ? represents the same equality sign or inequality sign, then for any $n >m$ it is true that $A(L,n)\ ?\ 0$ and $A(H,n)\ ?\ 0$.
\end{corollary}

For example, suppose $A = (1,-2)$. After checking that $L_3 - 2L_2 > 0$, $L_2 - 2L_1 > 0$ and $H_3 - 2H_2 > 0$ we can conclude the following.

\begin{corollary}\label{cor:2l-less-l}
For $n > 2$
\[ 2L_{n-1} < L_n \quad \text{ and } \quad 2H_{n-1} < H_n.\]
\end{corollary}

Similarly, for $A = (1,-3)$, after checking that $L_4 - 3L_3 < 0$, $L_3 - 3L_2 < 0$ and $H_4 - 3H_3 < 0$, we can conclude the following.

\begin{corollary}\label{cor:3l-more-l}
For $n > 3$
\[ 3L_{n-1} > L_n \quad \text{ and } \quad 3H_{n-1} > H_n.\]
\end{corollary}

The corollaries above are useful for showing the guaranteed growth rate of these sequences. The following corollary comparing the sequence $L_n$ to Jacobsthal numbers should prove useful as well.

\begin{corollary}\label{cor:LJ}
For $n > 4$, $L_n > J_{n+3}$.
\end{corollary}

\begin{proof}
The sequence $L_n$ grows faster than $J_n$ because $2L_{n-1} \leq L_n - 1$ and $2J_{n-1} \geq J_n -1$. We should also note that $L_{5} = 99 > J_{8} = 85$.
\end{proof}

\subsection{Itineraries and Strategies}.

As in the case of the light-real coin, the itineraries are not uniquely defined by the outcomes. We follow a method similar to the one used with the LR-coin to create itineraries. 

For the coins starting in the light state, we put each coin that just switched to the real state on the same pan again. Here is the description of the $i$-symbol in the itinerary.

\begin{itemize}
\item If the outcome has $<$ (correspondingly $>$) in the $i$-th place and this is an odd imbalance, the itinerary has L (correspondingly R).
\item If the outcome has $<$ (correspondingly $>$) in the $i$-th place and this is an even imbalance, the itinerary has R (correspondingly L).
\item If the outcome has $=$ in the $i$-th place and the previous place does not exist, or also has $=$ or an odd imbalance, the itinerary has O.
\item If the outcome has $=$ in the $i$-th place and an even imbalance in place $i-1$, then if the previous symbol is $<$, correspondingly $>$, the itinerary has R, correspondingly L.
\end{itemize}

For the coins starting in the heavy state, the rule is the same: we put each coin that just switched to the real state on the same pan, again, in the next weighing:

\begin{itemize}
\item If the outcome has $<$ (correspondingly $>$) in the $i$-th place, and this is an odd imbalance, the itinerary has R (correspondingly L).
\item If the outcome has $<$ (correspondingly $>$) in the $i$-th place, and this is an even imbalance, the itinerary has L (correspondingly R).
\item If the outcome has $=$ in the $i$-th place, and the previous place does not exist or also has $=$ or an even imbalance, the itinerary has O.
\item If the outcome has $=$ in the $i$-th place and an odd imbalance in place $i-1$, then, if the previous symbol is $<$, correspondingly $>$, the itinerary has R, correspondingly L.
\end{itemize}

For the coins starting in the real state, the rule is different: we put each coin on a pan before each odd occurrence of the imbalance:

\begin{itemize}
\item If the outcome has $<$ (correspondingly $>$) in the $i$-th place, and this is an odd imbalance, the itinerary has L (correspondingly R).
\item If the outcome has $<$ (correspondingly $>$) in the $i$-th place, and this is an even imbalance, the itinerary has R (correspondingly L).
\item If the outcome has $=$ in the $i$-th place, and the next place does not exist or also has $=$ or an even imbalance, the itinerary has O.
\item If the outcome has $=$ in the $i$-th place and an odd imbalance in place $i+1$, then, if the next symbol is $<$, correspondingly $>$, the itinerary has L, correspondingly R.
\end{itemize}

It is easy to see that conjugate outcomes generate conjugate itineraries per these rules. We use these itinerary assignments to prove the main theorem of this section which describes the exact bound.

\begin{theorem}
One LHR-coin, starting in the light state among $N$ coins, can be found in $w$ weighings, if and only if $N \leq L_w$. The coin starting in the real state can be found if and only if $N \leq R_{w} = L_{w-1}$. The coin starting in the heavy state can be found if and only if $N \leq H_w$. Moreover, if $N$ respects these bounds, there exists an oblivious strategy of finding the coin.
\end{theorem}

\begin{proof}
The number of coins cannot exceed the bound because the number of coins we process cannot exceed the possible number of outcomes.

The oblivious strategy is based on itineraries. Given the total number of coins, we assign the itineraries similar to the method we used for the light-heavy coin with a known starting state in Section~\ref{sec:lhknown}. If the number of coins is odd, then one of the coins is matched to the self-conjugate outcome. The rest of the coins are assigned in pairs to conjugate outcomes. This way, we are guaranteed to have itineraries in conjugate pairs, and therefore the corresponding strategy is legitimate.

Suppose this strategy produces an outcome $x$. For a given starting state, there is a bijection between outcomes and itineraries. That means that an outcome uniquely defines the existing itinerary and the coin.
\end{proof}

Before discussing the unknown state, we should like to discuss the mixed state, which will help us with the unknown state.

\section{Mixed known states}\label{sec:3statemixed}

Suppose all coins were divided into three disjoint groups: $l$ coins such that, if the fake coin is there, it must be in the light state; $h$ coins such that, if the fake coin is there, it must be in the heavy state; and $r$ coins such that, if the fake coin is there, it must be in the real state. We call this the \textit{mixed-known} state, or $l:h:r$ state.

\subsection{Outcomes}

We introduce four new sets of outcomes $\mathcal{LH}_n$, $\mathcal{HR}_n$, $\mathcal{LR}_n$, and $\mathcal{LHR}_n$, that describe the possible outcomes if the known state is limited to the initials of the sequences. In other words: $\mathcal{LH}_n = \mathcal{L}_n \cup \mathcal{H}_n$, $\mathcal{HR}_n = \mathcal{H}_n \cup \mathcal{R}_n$, $\mathcal{LR}_n = \mathcal{L}_n \cup \mathcal{R}_n$, and $\mathcal{LHR}_n = \mathcal{L}_n \cup \mathcal{H}_n \cup \mathcal{R}_n$. The sequences $LH_n$, $HR_n$, $LR_n$, and $LHR_n$ count the total number of outcomes in each set. For $n=0$, all the sequences are equal to 1.

We know that $\mathcal{R}_n \subset \mathcal{L}_n$, therefore $\mathcal{LR}_n = \mathcal{L}_n$, and $\mathcal{LHR}_n = \mathcal{LH}_n$. That is 
$$LR_n = L_n \quad \text{and} \quad LH_n = LHR_n.$$ The sets $\mathcal{L}_n$, $\mathcal{H}_n$, and $\mathcal{R}_n$ are represented as ellipses in the Venn diagram in Figure~\ref{fig:LHR}.

\begin{figure}[h]
    \centering
    \includegraphics[scale=0.3]{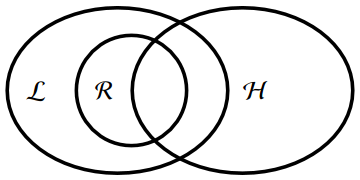}
    \caption{The Venn diagram of sets $\mathcal{L}_n$, $\mathcal{H}_n$, and $\mathcal{R}_n$.}
    \label{fig:LHR}
\end{figure}

We can express our new sequences at index $n$ as the function of themselves at index $n-1$:

\begin{lemma}
\begin{enumerate}
\item $HR_n = LHR_{n-1} + 2R_{n-1}$,
\item $LHR_n = LHR_{n-1} + 2HR_{n-1}$.
\end{enumerate}
\end{lemma}

\begin{proof}
Consider the $\mathcal{HR}$-outcomes. If the first sign is an imbalance, then the coin started in the heavy state, and after this it must be in the real state, so the total number of outcomes starting with an imbalance is $2R_{n-1}$. If the outcome starts with a balance, and the fake coin was not on the scale, then the fake coin is in the heavy or real state. Otherwise, the fake coin has been in the real state and on the scale, and now it is in the light state. Therefore, the coin can be in any state, and the number of such outcomes is $LHR_{n-1}$.

Consider the $\mathcal{LHR}$-outcomes. If the first sign is an imbalance, then the coin started in the heavy or light state, and after this it must be in the real or heavy state. The total number of outcomes starting with an imbalance is $2HR_{n-1}$. If the outcome starts with a balance, and the fake coin was on the scale, then the coin is now in the light state. If the fake coin was not on the scale, it can be in any state. Therefore, the coin corresponding to this outcome can now be in any state. The number of such outcomes is $LHR_{n-1}$.
\end{proof}

The sequence $LHR_n$ that counts the total number of outcomes is:

\[1,\ 3,\ 9,\ 19,\ 49,\ 123,\ 297,\ 707,\ 1697,\ 4043,\ \ldots.\]

The first seven terms of this sequence coincide with sequence A102001 in OEIS \cite{OEIS}, which consists of weighted Tribonacci numbers with  weights $(1,2,4)$. Sequence A102001 can also be defined as the number of strings of length $n$, made of symbols $<$, $>$, and $=$, where every three consecutive symbols contain as least one equal sign. For the strings up to length 6, this definition coincides with our definition. However, for seven weighings, the pattern $\neq \neq = \neq = \neq \neq$ would not be a legitimate outcome, although it would fit the sequence A102001. This is the only pattern that is excluded for length 7, which means $LHR_7 = \text{A102001}(7) - 32 = 739-32 = 707$. 

The sequence $HR_n$ is:

\[1,\ 3,\ 5,\ 15,\ 37,\ 87,\ 205,\ 497,\ \ldots.\]

For ease of visualization, we put all five sequences in Table~\ref{table:LHR}. 

\begin{table}[h!]
\centering
\begin{tabular}{| c |rrrrrrrr|} 
 \hline
$n$ & 0 & 1 & 2 & 3 & 4 & 5 & 6 & 7 \\ 
\hline
$R_n$ & 1 & 1 & 3 & 9 & 19 & 41 & 99 & 233\\
$H_n$ & 1 & 3 & 5 & 11 & 29 & 67 & 149 & 347\\
$HR_n$ & 1 & 3 & 5 & 15 & 37 & 87 & 205 & 497\\
$L_n$ & 1 & 3 & 9 & 19 & 41 & 99 & 233 & 531\\
$LHR_n$ & 1 & 3 & 9 & 19 & 49 & 123 & 297 & 707\\
 \hline
\end{tabular}
\caption{Sequences $R_n$, $H_n$, $HR_n$, $L_n$, and $LHR_n$.}
\label{table:LHR}
\end{table}

The sequences in the table seem to run in increasing order. The fact of $HR_n \geq H_n$ and $LHR_n \geq L_n$ follows from the fact that $\mathcal{H}_n \subset \mathcal{HR}_n$ and $\mathcal{L}_n \subset \mathcal{LHR}_n$. However, the sequence $HR_n$ switches sides. For $n <10$, we have $HR_n < L_n$; starting from $n=10$, it changes to $HR_n > L_n$.

Before producing the formulas for these sequences, we would like to introduce more sets of outcomes. Let us divide all outcomes into disjoint sets: $\mathcal{LX}$, $\mathcal{HX}$, $\mathcal{LHX}$, $\mathcal{LRX}$, and $\mathcal{LHRX}$. The letter X here stands for e\textbf{X}lusive. For example, the outcomes in $\mathcal{LHX}$ are such that they can be assigned to a coin in a light or a heavy state, but cannot be assigned to a coin in the real state. The outcomes that can match both the light and heavy state cannot have two imbalances in a row, and we will call them the \textit{alternator} outcomes, as they are the outcomes we discussed in the alternator (light-real) Section~\ref{sec:alternator}. The disjoint sets of outcomes are marked on the Venn diagram in Figure~\ref{fig:LHRX}.

\begin{figure}[h]
    \centering
    \includegraphics[scale=0.4]{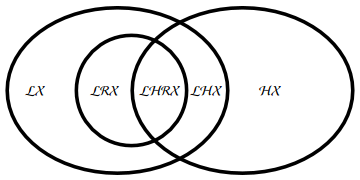}
    \caption{The Venn diagram of sets $\mathcal{LX}_n$, $\mathcal{HX}_n$, $\mathcal{LHX}_n$, $\mathcal{LRX}_n$, and $\mathcal{LHRX}_n$.}
    \label{fig:LHRX}
\end{figure}

In other words: 

\begin{itemize}
\item $\mathcal{LX} = \mathcal{L} \setminus (\mathcal{H} \cup \mathcal{R})$
\item $\mathcal{HX} = \mathcal{H} \setminus (\mathcal{L} \cup \mathcal{R})$
\item $\mathcal{LRX} = \mathcal{R} \setminus (\mathcal{LX} \cup \mathcal{H})$
\item $\mathcal{LHX} = (\mathcal{L} \cup \mathcal{H}) \setminus (\mathcal{LX} \cup \mathcal{HX} \cup \mathcal{R})$
\item $\mathcal{LHRX} = \mathcal{L} \cap \mathcal{H} \cap \mathcal{R}$.
\end{itemize}

The corresponding totals for these sets are $LX_n$, $HX_n$, $LHX_n$, $LRX_n$, and $LHRX_n$. The following lemma calculates these sequences. Obsessively, we can also introduce $\mathcal{RX}$ and $\mathcal{HRX}$ sets, but they are empty.

\begin{lemma}
The total number of outcomes in each X-group of length $n$ is:
\begin{itemize}
\item $LX_n = L_n - L_{n-1} - 2J_{n}$,
\item $HX_n = H_n - J_{n+2}$,
\item $LRX_n = L_{n-1} - J_{n+1}$,
\item $LHX_n = 2J_{n}$,
\item $LHRX_n = J_{n+1}$.
\end{itemize}
\end{lemma}

\begin{proof}
The outcomes that can only correspond to the light state start with an imbalance and do not correspond to an alternator outcome. That means that from the $\mathcal{L}$-outcomes we subtract the ones that start with a balance ($L_{n-1}$) as well as the ones that start with $\neq =$ and are followed by an alternator outcome ($2J_{n}$).                  

The outcomes which can only correspond to the heavy state do not correspond to an alternator outcome. That means that from the $\mathcal{H}$-outcomes we subtract $J_{n+2}$.

The number of outcomes that can be assigned to the coins in both the light and real states is $R_n$. Out of those, those that can be matched to a heavy state have to start with the balance and be followed by an alternator outcome ($J_{n+1}$).

The outcomes that can be assigned to the coins in both the light and heavy states, but not the real state, do not start with a balance, and also follow the alternator pattern.

The outcomes which can be assigned to any state start with a balance and follow the alternator pattern.
\end{proof}

Note that all groups save for the last one have an even number of outcomes.

Now that we know the number of outcomes in any disjoint group, we can count the number of outcomes in any group. For example, the $\mathcal{H}$-outcomes are the union of exclusive outcomes containing the letter $\mathcal{H}$: $\mathcal{H}_n = \mathcal{HX}_n \cup \mathcal{LHX}_n \cup \mathcal{LHRX}_n$. Therefore, $H_n = HX_n + LHX_n + LHRX_n$.

We are now ready to express $HR_n$ and $LHR_n$ via our three main sequences: $L_n$, $H_n$, and $J_n$.

\begin{lemma}\label{lemma:lhr-formulas}
\begin{enumerate}
\item $HR_n = H_n + L_{n-1} - J_{n+1}$,
\item $LHR_n = L_n + H_n - J_{n+2}$.
\end{enumerate}
\end{lemma}

\begin{proof}
The $\mathcal{HR}$-outcomes comprise the union of disjoint sets with at least one of the letters $\mathcal{H}$ or $\mathcal{R}$ in them: $\mathcal{HR}_n = \mathcal{HX}_n \cup \mathcal{LRX}_n \cup \mathcal{LHX}_n \cup \mathcal{LHRX}_n$. Therefore, $HR_n = H_n - J_{n+2} + L_{n-1} - J_{n+1} + 2J_{n} + J_{n+1} = H_n - J_{n+2} + L_{n-1}  + 2J_{n} = H_n + L_{n-1} - J_{n+1}$.

The $\mathcal{LHR}$-outcomes are the union of all the disjoint sets, and $L HR_n = L_n - L_{n-1} - 2J_{n} + H_n - J_{n+2} + L_{n-1} - J_{n+1} + 2J_{n} + J_{n+1} = L_n + H_n - J_{n+2}$.
\end{proof}

The following lemma states the recurrence for the sequences $LX_n$, $HX_n$, $LRX_n$, $HR_n$, and $LHR_n$.

\begin{lemma}
The sequences $LX_n$, $HX_n$, $LRX_n$, $HR_n$, and $LHR_n$ satisfy the recurrence 
$s_n = 3s_{n-1}-s_{n-2}+s_{n-3}-2s_{n-4}-8s_{n-5}$.
\end{lemma}

\begin{proof}
These sequences are linear combinations of sequences $L_n/H_n$ and $J_n$ that are recurrences with the characteristic polynomials $x^3-2x^2+x-4$ and $x^2 - x -2$, correspondingly. Therefore, the characteristic polynomial of their linear combination is the product of the polynomials \cite{genfunct}:
$$x^5-3x^4+x^3-x^2+2x+8.$$
\end{proof}

\subsection{Matching coins to outcomes}

We are now ready to produce a bound; we just need to define the following list of inequalities as $lhr_n$ inequalities.

\begin{equation} \label{eq-lhr-n}
\begin{split}
l & \leq L_n,\\
r & \leq R_n,\\
h & \leq H_n,\\
l+h & \leq LH_n,\\
h+r & \leq HR_n,\\
l+r & \leq LR_n,\\
l+h+r &\leq LHR_n.
\end{split}
\end{equation}

\begin{remark}
Two of the above inequalities are redundant. The first statement follows from the sixth statement. Indeed, $l \leq l+r \leq LR_n = L_n$. Similarly, the fourth statement follows from the last statement: $l+h \leq l+h+r \leq LHR_n = LH_n$. This is due to the fact that outcomes corresponding to the real state can also correspond to the light state.
\end{remark}

\begin{theorem}\label{thm:lhr-inequalities}
If the  $l:h:r$ state can be solved in $n$ weighings, then $l$, $h$, $r$, and $n$ satisfy $lhr_n$ inequalities.
\end{theorem}

\begin{proof}
The number of coins cannot be greater than the number of the possible outcomes corresponding to them.
\end{proof}

Now we shall match the coins to outcomes.

\begin{lemma}\label{lemma:lhr-coins2outcomes}
If the $lhr_n$ inequalities hold for numbers $l$, $h$, $r$, and $n$, then we can match the coins to the outcomes of length $n$ in such a way that each coin of a particular type is matched to an outcome that is allowed for this type.
\end{lemma}

\begin{proof}
We start by assigned $h$-coins moving from more-exclusive to less-exclusive sets. First, we assign as many coins as possible to $\mathcal{HX}_n$, then to $\mathcal{LHX}_n$, and finally to $\mathcal{LHRX}_n$. 

We use the same principle for the $l$-coins. First, we assign them to $\mathcal{LX}_n$ and $\mathcal{LHX}_n$, in order to leave as much space as possible for the $r$-coins. It does not matter in which order we use these two sets; let us say that we start with $\mathcal{LX}_n$.

After that, we assign the $l$-coins to the leftover $\mathcal{LRX}_n$ and $\mathcal{LHRX}_n$ groups in order. We need to prove that we do not run out of outcomes for the $l$-coins. Suppose $h \leq HX_n$. Then, neither of the $h$-coins impose on other coins, and we have at least $L_n$ outcomes available for the $l$-coins. Suppose $h - HX_n = x > 0$; then, $l+r \leq LR_n - x$, and we have $LR_n -x$ outcomes available.

We assign the leftover outcomes to the $r$-coins. We can assume that the order is $\mathcal{LRX}_n$ followed by $\mathcal{LHRX}_n$.

Again, we need to show that there are enough outcomes left. If the $l$-coins and $h$-coins do not spill over into $\mathcal{LRX}_n$ and $\mathcal{LHRX}_n$ groups, we have $R_n$ outcomes available for the $r$-coins. They can only spill over if $\mathcal{LX}_n$ is completely assigned. Also, if $\mathcal{LX}_n$, $\mathcal{HX}_n$ and $\mathcal{LHX}_n$ are completely used, the number of coins that is left over is not greater than the number of outcomes that are left over, and so we can match them. 

We are left with two cases: a) $l < LX_n$ and $h > HX_n + LHX_n$, or b) $l > LX_n+LHX_n$ and $h < HX_n$. Consider the first case. If $h - HX_n - LHX_n = x > 0$, then $r < R_n - x$, and the number of available outcomes is the same. The second case is similar. 
\end{proof}

The $lhr_n$ inequalities are not sufficient for a strategy to exist. Consider, for instance, the case of $k=1$ and $l=r=h=1$. The values $l$, $h$, and $r$ match the inequalities. On the other hand, a one-weighing strategy does not exist. The reason it does not work is that we have to assign the $=$ outcome to the real coin, and the imbalances to the $l$ and $h$ coins. Then, $l$ and $h$ coins need to be on the same pan in the first weighing, and we do not have other coins to balance them.

\subsection{An oblivious strategy}

Now we will show that $lhr_n$ inequalities are almost enough for an oblivious strategy to exist. Namely, if the inequalities hold and we have two spare genuine coins, then we can find an oblivious strategy.

\begin{theorem}\label{thm:obliviousexistence}
Suppose the $lhr_n$ inequalities hold for numbers $l$, $h$, $r$, and $n$. If $m$ of the numbers $l$, $h$, and $r$ are odd, then it is enough to have $m-1$ extra, genuine coins for an oblivious strategy with $n$ weighings to exists.
\end{theorem}

\begin{proof}
Per Lemma~\ref{lemma:lhr-coins2outcomes}, there is a way to assign coins to outcomes so that the outcomes match the coin types. Suppose we have such an assignment. Now we want to change some of these assignments in such a way that for one coin type a maximum of one X-group has an odd number of coins of this type. In addition, we want one odd group, if it exists, to be in $\mathcal{LHRX}$.

We start with $h$-coins. Suppose group $\mathcal{HX}$ and another group have an odd number of $h$-coins. In this case, the $\mathcal{HX}$ group has an outcome that is not assigned, and we can move an $h$-coin from another group to this group. Suppose groups $\mathcal{LHX}$ and $\mathcal{LHRX}$ both have an odd number of $h$-coins assigned. If group $\mathcal{LHX}$ have extra outcomes available, we can move one $h$-coin from $\mathcal{LHRX}$ to $\mathcal{LHX}$. If there are no extra outcomes available, then there must be $l$-coins in this group. In this case, we can swap an $l$-coin in $\mathcal{LHX}$ with and an $h$-coin in $\mathcal{LHRX}$. After this procedure, the statement is true for $h$-coins.

Now we shall look at $l$-coins. Similarly to the $h$-case, we can make sure that $\mathcal{LX}$ and $\mathcal{LHX}$ are the only sets containing an odd number of $l$-coins or an even number of $l$-coins. If they both have an even number of $l$-coins, we look further. Suppose both $\mathcal{LRX}$ and $\mathcal{LHRX}$ have an odd number of $l$-coins. Then, group $\mathcal{LRX}$ either has extra space available or at least one $r$-coin assigned. We can either move one $l$-coin to $\mathcal{LRX}$ from $\mathcal{LHRX}$ or swap it with an $r$-coin. Note that we only swap if a group is full. In this case, we do not increase the number of odd groups for $r$-coins. 

We process the $r$-coins in the same way. 

In the end, if we have one odd group and $\mathcal{LHRX}$ does not have odd group, then $\mathcal{LHRX}$ has outcomes available, and we can move one of the odd groups there. 

Now we shall assign itineraries. A coin from an odd group in $\mathcal{LHRX}$ is assigned a self-conjugate itinerary. Other coins are grouped into pairs of the same type within the same X-group. Each pair is assigned conjugate outcomes, and therefore conjugate itineraries. We have no more than $m-1$ extra coins left. We assign extra coins from odd groups to some outcomes and generate some itineraries for them. We use genuine coins to match these extra itineraries with conjugate ones in order to allow for a legitimate weighing strategy.

As the outcomes are uniquely defined by the type of coin and its self-itinerary, the strategy works.
\end{proof}

We immediately see that sometimes we do not need extra, genuine coins.

\begin{corollary}\label{cor:mixedoblivious}
Suppose the $lhr_n$ inequalities hold for numbers $l$, $h$, $r$, and $n$. If at least two out of three numbers $l$, $h$ and $r$ are even, then there exists an oblivious strategy that finds the fake coin in $n$ weighings in the $l:h:r$ state.
\end{corollary}

We find a bound, and it is exact:

\begin{theorem}
If we can find the fake coin out of $N$ coins in the mixed state, in $w$ weighings, then $N \leq LHR_w = L_w+H_w - J_{w+2}$. Moreover, if $N=LHR_w$, then there exists a mix of numbers $l$, $h$, $r$, where $l+h+r = N$ and the coin can be found, in $w$ weighings, in an oblivious strategy.
\end{theorem}

\begin{proof}
The upper bound follows from Theorem~\ref{thm:itb}. To prove such an existence, we need to find an instance of $l$, $h$, and $r$ values satisfying the inequality conditions, with $l$ and $h$ being even. Here it is: $l = LX_w + LHX_w$, $h = HX_w$, and $r= LRX_w + LHRX_w$.
\end{proof}

For the unknown starting state, below, we need to pay special attention to the case $0:h:r$, that is when $l=0$. We already know that if at least one of $h$ and $r$ is even, and they satisfy the inequalities, then the oblivious strategy exists. We can also find an oblivious strategy in the following additional case.

\begin{lemma}\label{lemma:0hr}
Suppose the inequalities hold, and both $h$ and $r$ are odd. If $h < H_w$, then there exists an oblivious strategy of $w$ weighings.
\end{lemma}

\begin{proof}
When we assign $h$-coins to disjoint sets of outcomes, we can always make sure that the odd number of $h$-coins is either in $\mathcal{LHX}_w$, or in $\mathcal{LHRX}_w$ if $\mathcal{HX}_w$ and $\mathcal{LHX}_w$ are full. 

Given that $h < H_w$, there will be space left in $\mathcal{LHRX}_w$ for the $r$-coins. And so, we place the odd number of $r$-coins into $\mathcal{LHRX}_w$. 

Now we assign the outcomes in conjugate pairs to pairs of $h$- and $r$-coins. We use the self-conjugate itinerary for the extra $r$-coin, and either $<===\ldots$ or $=<==\ldots$ for the extra $h$-coin. We give the $h$-coin the itinerary, so it is on the scale exactly once, matching the imbalance. We assign the conjugate itinerary to the extra $r$-coin.
\end{proof}

\subsection{Examples}

\textbf{An adaptive strategy exists, but an oblivious one does not}

Consider the case when $l=7$, $h=1$, and $r=1$. In the first weighing of an adaptive strategy, we compare three $l$-coins against another three $l$-coins. If the weighing unbalances, then we can find one coin out of three on the light pan that will be in the heavy state for the next weighing. If the first weighing balances, we have the case of $l=h=r=1$, with 6 extra coins proven to be genuine, that can be solved by Theorem~\ref{thm:obliviousexistence}.

Now we want to prove that there is no oblivious strategy. The outcomes are divided into exclusive sets as follows: 

\begin{itemize}
\item $\mathcal{LX}$: $<<$, $<>$, $><$, $>>$,
\item $\mathcal{LHX}$: $<=$, $>=$
\item $\mathcal{LHRX}$: $=<$, $=>$, $==$.
\end{itemize}

If the $r$-coin is assigned to either $=<$ or $=>$, it must be on the scale in the first weighing, which is a contradiction, as we get an odd number of coins on the scale. Therefore, it must be assigned to $==$. 

The first group must consist of $l$-coins. Suppose the second group is also assigned to the $l$-coins. These coins must have itineraries LL, LR, RL, RR, LO, RO. Without loss of generality, we can assign $=<$ to the $l$-coin, and $=>$ to the $h$-coin. Then their itineraries must be OL and OL. We see that the second weighing has two more coins on the left pan. 

Now, suppose the second group consists of one $l$-coin and one $h$-coin. Without loss of generality, we have the following itineraries: LL, LR, RL, RR, LO, LO for the coins that participate in the first weighing, and again we have two more coins on one of the pans.

\section{A Light-Heavy-Real Coin: Unknown Starting State}\label{sec:3stateunknown}

The information-theoretical bound from Theorem~\ref{thm:itb} shows  that the number of coins we can process in $w$ weighings is not more than $(LHR_w +2)/3$. That means the following sequence gives the bound (zero-indexed): 

\[1,\ 1,\ 3,\ 7,\ 17,\ 41,\ 99,\ 236,\ \ldots.\]

Now we use parity considerations to refine our bound. Consider the first weighing, in which we put $2k$ coins on each pan. If the weighing unbalances, each of the coins on the scale has two outcomes that correspond to it, which start with an imbalance, depending whether the coin was on the lighter or the heavier pan. That means we can use no more than $4k$ outcomes which start with an imbalance. We leave it to the reader to prove that sequence $LHR_n$ modulo 4 is the sequence that alternates between 1 and 3. The total number of outcomes of length $n$ that start with an imbalance is $LHR_n - LHR_{n-1}$. This number is divisible by 2, but not by 4. Therefore we have two outcomes that we cannot assign. Keeping in mind that for zero and 1 weighings this argument does not work, as we do not need to put anything on the scale, our new bound is $LHR_w/3$, for $w > 1$: 

\[1,\ 1,\ 3,\ 6,\ 16,\ 41,\ 99,\ 235,\ \ldots.\]

We can check that with one weighing we can process not more than one coin. With two weighings we can process three coins by comparing the first and the second coin twice. Later, we will show an oblivious strategy that resolves 6 coins in three weighings, and an adaptive strategy that resolves 16 coins in four weighings. Starting from $n=5$, this bound is not achievable, as we shall see very soon.

\subsection{Unknown State. Adaptive Strategy}

Consider $u$ coins in the unknown state, which we want to solve in $w$ weighings. We denote this state as $0:0:0:u$. More generally, $l:h:r:u$ means we have $l$ coins, such that if one of them is fake, it must start in the light state---and similarly for other letters. We are trying to find the largest number of coins that we can process in $w$ weighings.

Let us look at the first weighing, after which we will know the state of all the coins on the scale. Suppose we have $x$ coins on each pan. If the weighing is unbalanced, then the fake coin was in either the light state on the lighter pan or in the heavy state on the heavier pan. After the first unbalanced weighing, we have a mixed state, as $0:x:x$. 

\begin{lemma}\label{lemma:0kk}
If $k> \min\{(HR_{w}-1)/2,R_w\}$ the case $0:k:k$ cannot be solved in $w$ weighings. Moreover, if $k \leq \min\{(HR_{w}-1)/2,R_w\}$, there exists an oblivious strategy that solves it in $w$ weighings.
\end{lemma}

\begin{proof}
The first statement follows from Theorem~\ref{thm:lhr-inequalities} and the inequalities $k \leq R_w$ and $k+k \leq HR_w$, keeping in mind that $HR_w$ is odd. For the second one, we need to remember that $R_w < H_w$ (for $w > 0$), and use Lemma~\ref{lemma:0hr}. For $w=0$, we have $k \leq \min\{(0,1\}$, and the statement is true.
\end{proof}

Consider the sequence $k_w = \min\{(HR_{w}-1)/2,R_w\}$ that provides the exact bound for solving the $0:k:k$ case in $w$ weighings:
\[0,\ 1,\ 2,\ 7,\ 18,\ 41,\ 99,\ 233,\ \ldots,\]

Starting from the fifth index, the sequence continues as $R_n$, as the following lemma proves.

\begin{lemma}\label{lemma:kn}
$2R_n < HR_n$ starting from $n \geq 5$.
\end{lemma}

\begin{proof}
The inequality $2R_n < HR_n$ is equivalent to $H_n - L_{n-1}  > J_{n+1}$. We first prove that $H_n - L_{n-1}  > L_{n-2}$, for any $n$. According to Lemma~\ref{lemma:AB}, it is enough to show that $H_2 = 5 > L_1 + L_0 = 4$ and $L_2 = 9 > H_2+H_1 = 3+5$. Now we can manually check the statement for $n=5,6$ and remember that $L_{n-2} > J_{n+1}$ from Corollary~\ref{cor:LJ}.
\end{proof}

Now we know our first weighing.

\begin{lemma}
There exists an adaptive strategy that processes the maximum number of coins in $w$ weighings, such that the first weighing has $k_{w-1}$ coins on each pan.
\end{lemma}

\begin{proof}
Per Lemma~\ref{lemma:0kk}, the mixed state $0:x:x$ can be solved in $w$ weighings, if and only if $x \leq k_w$. That means we cannot put more than $k_{w-1}$ coins on each pan in the first weighing.

If the first weighing balances, we have, after the weighing, $2x$ coins in the light state and $u-2x$ coins in the unknown state. The total is $u$, and the number of weighings needed for $l_1:0:0:u-l_1$ is no more than the number of weighings needed for $l_2:0:0:u-l_2$, when $l_1 > l_2$. That means we should put as many coins on the scale in the first weighing as the first weighing allows. 

If the first weighing is balanced, we get the $2k_{w-1}:0:0:u-2k_{w-1}$ state to process. Note that $2k_{w-1} \leq 2R_{w-1} = 2L_{w-2}$. In $w-1$ weighings, we can process up to $L_{w-1}$ of $l$-coins. Per Lemma~\ref{cor:2l-less-l}, we can see that $2R_{w-1} < L_{w-1}$, and in general the $L_n$ sequence grows approximately as a geometric progression with a coefficient of 2.3. That means we can process $2k_{w-1}$ of $l$-coins in $w-1$ weighings, and have a little bit of extra room for more coins.
\end{proof}

\textbf{Example $w=4$}. We have already shown that we cannot process more than 16 coins in four weighings. Now, we will show an adaptive strategy for 16 coins. In the first weighing, we put 7 coins on each pan. If the weighing unbalances, we have a $0:7:7$ state. This case is solvable in two weighings via Lemma~\ref{lemma:0kk}. If the first weighing balances, we have only two coins in the unknown state and 14 coins in the light state after the weighing. 

For the second weighing, we put 5 coins in the light state on each pan. If the weighing unbalances, we have 5 coins in the heavy state, which we know how to resolve in two weighings. If the weighing balances, we have 4 leftover coins in the light state and 2 coins in the unknown state under suspicion. For the third weighing, we compare two $l$-coins and one $u$-coin on the left pan, and two $l$-coins and one genuine coin on the right pan. If the weighing unbalances, then the fake coin is one of the two $l$-coins on the lighter pan or an unknown coin on one of the pans. We can find the fake coins by comparing two former $l$-coins that are now in the heavy state. If the weighing balances, then the fake coin is one of the former unknown coins. Moreover, if the fake coin was on the scale, then it is currently in the light state, and comparing it with a genuine coin finds the fake coin.

Back to any number of coins. Now we know what we should in the first weighing, and also what to do if it unbalances. We proceed to the second weighing after the first balance. We assume that $w> 4$.

\begin{lemma}
There exists an adaptive strategy that processes the maximum number of coins in $w > 4$ weighings, such that the second weighing after the first balance has $H_{w-2}$ coins on each pan, and the coins that were not on the scale in the first weighing are as evenly distributed between pans as possible. Moreover, any second weighing described here, if it unbalances, can be processed in $w-2$ weighings.
\end{lemma}

\begin{proof}
First we show that such a weighing is possible. For $w > 4$ we have $k_{w-1} > H_{w-2}$. We can check this for $w=5$, and for $w > 5$ it follows from the fact that $k_n = R_n = L_{n-1} > H_{n-1}$. That means we have $2H_{w-2}$ coins available to put on the scale. 

Now we study the weighing. We have $2k_{w-1}$ of $l$-coins and $x \leq L_{w-1} - 2k_{w-1}$ of $u$-coins before this weighing. Suppose we put $u_1$ of $u$-coins on the left pan and $u_2$ of $u$-coins on the right pan, for a total of $u_t$ of $u$-coins on the scale. Suppose that the total number of coins on each pan is $t$. After a balance, the state is $2k_{w-1} + u_t - (2t - u_t):0:0:x-u_t$, or $2k_{w-1} - 2t + 2u_t:0:0:x-u_t$. That means that after a balance it does not matter how the $u$-coins were distributed on the scale, and it is preferable to increase the total number of coins by adding to the first weighing as many $l$-coins as the imbalance in this weighing would allow.

Let us now look at imbalances. If the weighing unbalances, we know that either one of the coins on the lighter pan was fake and in the light state, or that one of the $u$-coins on the heavier pan was fake and in the heavy state. After the weighing, we have a $0:t:u_i:0$ state. Per Theorem~\ref{thm:lhr-inequalities}, this case is solvable only if $t \leq H_{w-2}$, $u_i \leq R_{w-2}$ and $t + u_i \leq HR_{w-2}$. We see that it is beneficial to redistribute those $u$-coins as evenly as possible, as only the $\max\{u_1,u_2\}$ plays a role. We also see that it is acceptable to add as many $l$-coins as the equations allow.

Now we assume that the $u$-coins are almost evenly distributed on the pans, and see how they influence inequalities. Let us assume that we added $l$-coins to both pans so that the total is $H_{w-2}$. We need to show that the following inequalities hold: 
\[H_{w-2} \leq H_{w-2}, \quad u_i \leq R_{w-2} \quad \text{and} \quad H_{w-2} + u_i \leq HR_{w-2}.\] 

The first inequality is trivial.

For the second inequality, we observe that, after the first weighing, the number of coins in the unknown state is not more than $L_{w-1} - 2k_{w-1}$, which is less than $2R_{w-2}$. Again we check $w=5$ manually, and for $w > 5$ it follows from Corollary~\ref{cor:A} and the fact that $L_4 - 2L_{3} \leq 2L_{2}$, $L_3 - 2L_{2} \leq 2L_{1}$, and $H_4 - 2H_{3} \leq 2H_{2}$. That means that, even if we put all the $u$-coins on the scale, the second inequality still holds. As $R_{w-2} < H_{w-2}$, all the $u$-coins can be put on the scale while satisfying all the inequalities.

For the third inequality, let us use the fact that $u_i \leq (L_{w-1}+1)/2 - k_{w-1}$. We need to show that $H_{w-2} + (L_{w-1}+1)/2 - k_{w-1} \leq HR_{w-2}$. We can check it for $w=5$. For $w > 5$, this is equivalent to 
$H_{w-2} + (L_{w-1}+1)/2 - R_{w-1} \leq HR_{w-2}$. That is, we need to show that 
$$2H_{w-2} + L_{w-1}+1 - 2L_{w-2} \leq 2HR_{w-2} = 2H_{w-2} + 2L_{w-3} - 2J_{w-1}.$$ 
Or
$$ L_{w-1}+1 + 2J_{w-1} \leq  2L_{w-3} + 2L_{w-2}.$$ 

We know that $J_{w-1} < L_{w-4}$ by Corollary~\ref{cor:LJ}. That means it is enough to show 
$$ L_{w-1}+ 2L_{w-4} <  2L_{w-3} + 2L_{w-2}.$$ 
We can use the same technique we used before, so we need only to check that the following three statements are true: $L_{3}+ 2L_{0} <  2L_{1} + 2L_{2}$, $L_{4}+ 2L_{1} <  2L_{2} + 2L_{3}$, and $H_{4}+ 2H_{1} <  2H_{2} + 2H_{3}$. And they are true.

We conclude that for $w > 4$ we can put any number of $u$-coins on the scale, distribute them as evenly as possible, add more $l$-coins, so that the total is $H_{w-2}$; then, if the weighing imbalances, all the necessary inequalities hold. As we have at least two extra, genuine coins outside the scale, per Theorem~\ref{thm:obliviousexistence} an oblivious strategy exists that solves the $0:H_{w-2}:u_i:0$ state produced by an imbalance.
\end{proof}

How many $u$-coins should we put on the scale in the second weighing? What happens if we remove a $u$-coin from the scale and replace it with an $l$-coin? This would not make things worse for an imbalance. For a balance, the number of coins left to process decreases by 1, but the state of one coin changes from light to unknown. This means that, if after the balance we have the $a:0:0:b$ state, then after the swap we get the $a-2:0:0:b+1$ state. What is better? Here we present a lemma:

\begin{lemma}\label{lemma:l00u}
The state $l:0:0:u$ cannot be solved in $n$ weighings if $l+2u-1 > L_n$.
\end{lemma}

\begin{proof}
We use induction. The base of induction is $n=1$. Consider the state $l:0:0:u$. We know that $l \leq 3$. That means we need to manually check the following cases:
\begin{itemize}
\item If $l=0$, then $u > 2$. The state $0:0:0:3$ cannot be solved.
\item If $l=1$, then $u > 1$. The state $1:0:0:2$ cannot be solved.
\item If $l=2$, then $u > 1$. The state $2:0:0:1$ cannot be solved.
\item If $l=3$, then $u > 0$. The state $3:0:0:1$ cannot be solved.
\end{itemize}

Now, the step of the induction. Suppose we already proved the bound $l+2u-1 \le L_n$ for $n$ weighings.

Suppose that, as before, on the pans we place $l_1$ and $l_2$ coins of $l$-type, and $u_1$ and $u_2$ coins of $u$-type correspondingly. Then, the following inequalities hold:

\begin{itemize}
\item $l_1+u_1 \le HR_{n}$ (so we can solve in $n$ weighings the case when the left pan is lighter),
\item $l_2+u_2 \le HR_{n}$ (so we can solve in $n$ weighings the case when the left pan is heavier),
\item $(l-l_1-l_2)+u_1+u_2 +2(u-u_1-u_2)-1 \le L_n$ (The state after the balance is $(l-l_1-l_2)+u_1+u_2 :0:0:u-u_1-u_2$, and the inequality is the induction assumption).
\end{itemize}

Summing up, we get $l +2u -1 \le L_n + 2HR_n = L_{n+1}$. The lemma is proven.
\end{proof}

Now we present a particular adaptive strategy, in which we put all but one $u$-coins on the scale after the first weighing balances. We also show that this strategy maximizes the number of processed coins.

\begin{theorem}
For $w > 4$, there exists an adaptive strategy of $w$ weighings for the state with $N$ unknown coins, if and only if $N \leq k_{w-1} + (L_{w-1}+1)/2$.
\end{theorem}

\begin{proof}
We only need to discuss the second weighing after a balance. After the first weighing, the state is $2k_{w-1}:0:0:N-2k_{w-1}$. In the second weighing, we put all but one $u$-coins on the scale. We already showed that we need to add $l$-coins to the scale, so that the total is $H_{w-2}$ on each pan, and distribute the $u$-coins as evenly as possible. We have also shown that the imbalance can be resolved in $w-2$ weighings.

Suppose the weighing is a balance. After it we will have $2k_{w-1} - 2H_{w-2} + 2(N-2k_{w-1})-2:0:0:1$ state. Here we have one $u$-coin and many $l$-coins. We can assign the self-conjugate itinerary to the $u$-coin, which means that we can process this state if and only if the total number of coins does not exceed $L_{w-2}$:
\[ 2k_{w-1} - 2H_{w-2} + 2(N-2k_{w-1})-1 \leq L_{w-2},\]
equivalently
\[ 2k_{w-1} + 2(N-2k_{w-1})-1 \leq L_{w-1}.\]
Equivalently
\[N \leq k_{w-1} + (L_{w-1}+1)/2.\]

We cannot increase the number of $u$-coins by Lemma~\ref{lemma:l00u}.
\end{proof}

We showed that the unknown state is solvable if and only if the number of coins is bounded by the sequence 
\[1,\ 1,\ 3,\ 6,\ 16,\ 39,\ 91,\ 216,\ 499,\ 1144,\ 2651,\ \ldots,\]
which equals $k_{w-1} + (L_{w-1}+1)/2$ starting from $w = 5$.

Now we will present an adaptive strategy for 39 coins. 

\textbf{Example $w=5$}. In the first weighing, we put $18 = k_4$ coins on each pan. If the weighing unbalances, we have a $0:18:18$ state. This case is solvable in three weighings, via Lemma~\ref{lemma:0kk}. If the weighing balances, we have only three coins in the unknown state and 36 coins in the light state after the weighing: $36:0:0:3$ state. 

For the second weighing, we put ten $l$-coins and one $u$-coin on each pan. If the weighing unbalances, we have a $0:11:1:0$ state, which we know how to resolve in three weighings. If the weighing balances, we have a $18:0:0:1$ state, which we also know how to solve in three weighings.

\subsection{Unknown State. Oblivious Strategy}

Let us look at a few small examples. We can leave it to the reader to see that in one weighing we can process up to 1 coin and in 2 weighings---up to 3 coins. 

Here we show an oblivious strategy for 6 coins in three weighings, offering an example of itineraries for 6 coins: LLL, LRO, ORR, RLR, ROL, and OOO. The following Table~\ref{table:3wui2o} matches the itineraries to the outcomes:

\begin{table}[h!]
\centering
\begin{tabular}{| c | c | c | c|} 
 \hline
  & light & heavy & real \\ 
LLL    &      $<>=$    &      $>=<$ & $=<>$\\
LRO     &     $<<=$    &      $>==$ & $=>=$\\
ORR     &     $=><$    &      $=<=$ & $==>$ \\
RLR     &     $>>=$    &      $<=>$ & $=<<$\\
ROL     &     $>=>$    &      $<==$ & $==<$\\
OOO   	&     $===$    &      $===$ & $===$\\
 \hline
\end{tabular}
\caption{Matching of itineraries to outcomes for three weighings}
\label{table:3wui2o}
\end{table}

Note that it is possible to match the three unused outcomes: $><=$, $<=<$, and $=>>$ to itinerary RRL in the light, heavy and real state correspondingly, but we cannot add an extra coin, as we cannot balance all these seven itineraries. If we had one coin which we knew to be real, we could use it to balance out this seventh coin. In any case, we already know that an adaptive strategy for seven coins does not exists: the previous discussion is for purposes of illustration only.

We can also process five coins in three weighings. We can use the same set of itineraries without OOO. For 4 coins we can use itineraries LLL, LRR, RLR, and RRL. We leave it to the reader to check that this set of itineraries works.

We have so far been able to produce an oblivious strategy for every adaptive strategy. This changes with 4 weighings. We know that an adaptive strategy exists for 16 coins. Let us look at oblivious strategies. 

There are 49 possible outcomes. Consider outcome pattern $\neq \neq = \neq$. It must correspond to a coin in the light state that is on the scale at every weighing. The same coin in the heavy state will have outcome pattern $\neq = \neq \neq$ and in the real state: $=\neq \neq =$. There are 8 possible outcomes available for the heavy and light state, and only four for the real state. That means that these 20 outcomes can serve no more than 4 coins. We have 29 outcomes left that can be used for no more than 10 coins. Thus, an oblivious strategy can process no more than 14 coins.

We see that, starting at 4 weighings, adaptive strategies are more powerful than oblivious strategies.

\appendix

\section{Light-real coin, unknown state}

\subsection{Five weighings}\label{app:al5w20c}

Let us show how to build a strategy for finding the light-real coin in the unknown state, with 20 coins total.

Four pattern pairs $\neq = \neq = \neq$ and $= \neq = \neq =$ are matched to itineraries LLLLL, LLRRR, RRLLL, RRRRR.

The next four patten pairs: $\neq = \neq = =$ and $= \neq == \neq$ are matched to itineraries LLLOL, LLROR, RRLOL, RRROR.

Similarly, pairs $\neq = = \neq = $ and $== \neq = \neq$ are matched to itineraries LOLLL, LOLRR, RORLL, RORRR.

The next case is interesting: $\neq = = = \neq$ is matched to $= \neq ===$ or $=== \neq =$. The pattern for the light state is split between two different patterns for the real state. The corresponding itineraries are LLOOL, RROOR, LOOLL and ROORR. They are, again, in conjugate pairs.

We are left with three patterns, $\neq ====$, $==\neq ==$, and $==== \neq$. We split them into three pairs that we will match: $\neq ====$ and $==\neq ==$, $\neq ====$ and $==== \neq$, $==\neq ==$, and $==== \neq$. We obtain the itineraries: LOLOO, ROOOL, and OOROR. These itineraries balance out as well.

Finally, we have the OOOOO itinerary.

These are the itineraries for an oblivious strategy for 20 coins.

We have one self-conjugate itinerary and many conjugate pairs, which means that we can process any number of coins below 20.

\subsection{Six weighings}\label{app:al6w41c}

As before, we pair outcome patterns, then assign balanced itineraries to them. In Table~\ref{table:a6w} we show patterns of outcomes corresponding to the light and real state, then we find the itineraries. The last column counts the number of coins.

\begin{table}[h!]
\centering
\begin{tabular}{| c | c | c | c |} 
 \hline
light & real & balanced itineraries & \\
\hline
$\neq = \neq = \neq = $ & $= \neq = \neq = \neq$ & \begin{tabular}{@{}c@{}}LLLLLL LLLLRR LLRRLL LLRRRR \\ RRLLLL RRLLRR RRRRLL RRRRRR\end{tabular}  & 8\\ 
$\neq = \neq == \neq $ & $= \neq = \neq ==$ &  LLLLOL LLRROR RRLLOL RRRROR & 4\\ 
$\neq = \neq == \neq $ & $= \neq == \neq =$ &  LLLOLL LLRORR RRLOLL RRRORR & 4\\ 
$\neq == \neq = \neq $ & $== \neq = \neq =$ &  LOLLLL LOLRRR RORLLL RORRRR & 4\\ 
$\neq = \neq === $ & $= \neq ===\neq$ & LLLOOL LLROOR RRLOOL RRROOR & 4\\ 
$\neq === \neq = $ & $=== \neq =\neq$ & LOOLLL LOOLRR ROORLL ROORRR & 4\\
$\neq == \neq == $ & $== \neq ==\neq$ & LOLLOL LOLROR RORLOL RORROR & 4\\
$\neq ==== \neq$ & $= \neq ====$ and $==== \neq =$ & LLOOOL RROOOR LOOOLL ROOORR & 4\\
$\neq =====$ & $== \neq ===$ & LOLOOO ROROOO & 2\\
$==\neq ===$ & $===== \neq$ & OOOLOL OOOROR & 2\\
$======$ & $======$ & OOOOOO & 1\\ 
 \hline
\end{tabular}
\caption{Six weighings}
\label{table:a6w}
\end{table}

These are the itineraries for an oblivious strategy for 41 coins.

We have one self-conjugate itinerary, and all other itineraries are in conjugate pairs, which means that we can process any number of coins below 41.

\subsection{Seven weighings}\label{app:al7w82c}

Note that in many cases we build itineraries in such a way that on-scale letters are in groups of two repeated letters. If the number of on-scale letters is odd, then the last on-scale letter matches the previous on-scale letter.

Table~\ref{table:a7w} shows how to process 82 coins in seven weighings. In the table we do not explicitly write itineraries that are easy to construct.

\begin{table}[h!]
\centering
\begin{tabular}{| c | c | c | c |} 
 \hline
light & real & balanced itineraries & \\
\hline
$\neq = \neq = \neq = \neq $ & $= \neq = \neq = \neq =$ &   & 8\\ 
$\neq = \neq = \neq == $ & $= \neq = \neq == \neq$ &   & 8\\ 
$\neq = \neq == \neq =$ & $= \neq == \neq = \neq $ &   & 8\\ 
$\neq == \neq = \neq =$ & $== \neq = \neq = \neq $ &   & 8\\ 
$\neq = \neq === \neq$ & $= \neq === \neq = $ and $= \neq = \neq === $ &   & 8\\ 
$\neq == \neq == \neq $ & $== \neq = \neq == $ and $= \neq == \neq == $ &   & 8\\ 
$\neq === \neq = \neq $ & $=== \neq = \neq = $ and $== \neq == \neq = $ &   & 8\\ 
$\neq = \neq ==== $ & $= \neq ====\neq$ &  & 4\\ 
$\neq == \neq === $ & $== \neq ===\neq$ &  & 4\\
$\neq === \neq == $ & $=== \neq ==\neq$ &  & 4\\
$\neq ==== \neq=$ & $==== \neq = \neq$ &  & 4\\
$\neq ===== \neq$ & $= \neq =====$ and $===== \neq =$&  & 4\\
$\neq ======$ & $====== \neq $  &  & 2\\
$==\neq ====$ & $=== \neq ===$  & OOLLOOO & 1\\
$===\neq ===$ & $==== \neq==$ & OOORROO & 1\\
$==\neq ====$ & $==== \neq==$  & OOROLOO & 1\\
$=======$ & $=======$ & OOOOOOO & 1\\ 
 \hline
\end{tabular}
\caption{Seven weighings}
\label{table:a7w}
\end{table}

\end{document}